\documentclass{article}%
\usepackage{amsmath}
\usepackage{graphicx}
\usepackage{amsfonts}
\usepackage{amssymb}%
\newtheorem{theorem}{Theorem}

\newtheorem{corollary}[theorem]{Corollary}

\newtheorem{definition}[theorem]{Definition}

\newtheorem{lemma}[theorem]{Lemma}

\newtheorem{proposition}[theorem]{Proposition}
\newtheorem{remark}[theorem]{Remark}

\newenvironment{proof}[1][Proof]{\textbf{#1.} }{\ \rule{0.5em}{0.5em}}
\begin{document}

\title{General upper and lower tail estimates using Malliavin calculus and Stein's equations}
\author{Richard Eden$^{\dagger}$\\Department of Mathematics,
\and Frederi Viens\thanks{\texttt{viens@purdue.edu}\ \ \emph{corresponding author}%
}~$^{,}$\thanks{Both authors' research partially supported by NSF grant
DMS-0907321.}\\Department of Statistics
\and \vspace*{-0.2in}\\Purdue University\\150 N. University St.\\West Lafayette, IN 47907-2067}
\date{\ }
\maketitle

\begin{abstract}
Following a strategy recently developed by Ivan Nourdin and Giovanni Peccati,
we provide a general technique to compare the tail of a given random variable
to that of a reference distribution. This enables us to give concrete
conditions to ensure upper and/or lower bounds on the random variable's tail
of various power or exponential types. The Nourdin-Peccati strategy analyzes
the relation between Stein's method and the Malliavin calculus, and is adapted
to dealing with comparisons to the Gaussian law. By studying the behavior of
the solution to general Stein equations in detail, we show that the strategy
can be extended to comparisons to a wide class of laws, including many Pearson distributions.

\end{abstract}



\section{Introduction}

In this article, following a strategy recently developed by Ivan Nourdin and
Giovanni Peccati, we provide a general technique to compare the tail of a
given random variable to that of a reference distribution, and apply it for
all reference distributions in the so-called Pearson class, which enables us
to give concrete conditions to ensure upper and/or lower bounds on the random
variable's tail of power or exponential type. The strategy uses the relation
between Stein's method and the Malliavin calculus. In this introduction, we
detail the main ideas of this strategy, including references to related works;
we also summarize the results proved in this article, and the methods used to
prove them.

\subsection{Stein's method and the analysis of Nourdin and Peccati}

Stein's method is a set of procedures that is often used to measure distances
between distributions of random variables. The starting point is the so-called
Stein equation. To motivate it, recall the following result which is sometimes
referred to as Stein's lemma. Suppose $X$ is a centered random variable. Then
$X\overset{\text{Law}}{=}Z\thicksim\mathcal{N}(0,1)$ if and only if
\begin{equation}
\mathbf{E}[f^{\prime}(X)-Xf(X)]=0\label{stein_normal}%
\end{equation}
for all continuous and piecewise differentiable functions $f$ such that
$\mathbf{E}[|f^{\prime}(X)|]<\infty$ (see e.g. \cite{CS}, \cite{DZ},
\cite{Stein}). If the above expectation is non-zero but close to zero, Stein's
method can give us a way to express how close the law of $X$ might be to the
standard normal law, in particular by using the concept of Stein equation. For
a given test function $h$, this is the ordinary differential equation
$f^{\prime}\left(  x\right)  -xf\left(  x\right)  =h\left(  x\right)
-\mathbf{E}\left[  h\left(  Z\right)  \right]  $ with continuous and piecewise
differentiable solution $f$. As we will see in more detail and greater
generality further below, if one is able to prove boundedness properties of
$f$ and $f^{\prime}$ for a wide class of test functions $f$, this can help
evaluate the distance between the law of $Z$ and laws of random variables that
might be close to $Z$, including methods for proving convergence in
distribution. This fundamental feature of Stein's method is described in many
works; see \cite{CS} for a general introduction and review.

As a testament to the extraordinary versatility of Stein's method, recently
Ivan Nourdin and Giovanni Peccati discovered a connection between Stein's
method and the Malliavin calculus, with striking applications in a number of
problems in stochastic analysis. Motivated by Berry-Ess\'{e}en-type theorems
for convergence of sequences of random variables in Wiener chaos, Nourdin and
Peccati's first paper \cite{NP} on this connection considers an arbitrary
square-integrable Malliavin-differentiable random variable $X$ on a Wiener
space, and associates the random variable%
\begin{equation}
G:=\langle DX;-DL^{-1}X\rangle\label{Gee}%
\end{equation}
where $D$ is the Malliavin derivative operator on the Wiener space, and
$L^{-1}$ is the pseudo-inverse of the generator of the Ornstein-Uhlenbeck
semigroup (see Section \ref{MALL} for precise definitions of these operators).
One easily notes that if $X$ is standard normal, then $G\equiv1$. Then by
measuring the distance between $G$ and $1$ for an arbitrary $X$, one can
measure how close the law of $X$ is to the normal law. The connection to
Stein's method comes from their systematic use of the basic observation that
$\mathbf{E}\left[  Gf\left(  X\right)  \right]  =\mathbf{E}\left[  Xf^{\prime
}\left(  X\right)  \right]  $. It leads to the following simple and efficient
strategy for measuring distances between the laws of $X$ and $Z$. To evaluate,
e.g., $\mathbf{E}\left[  h(X)\right]  -\mathbf{E}\left[  h(Z)\right]  $ for
test functions $h$, one can:

\begin{enumerate}
\item write $\mathbf{E}\left[  h(X)\right]  -\mathbf{E}\left[  h(Z)\right]  $
using the solution of Stein's equation, as $\mathbf{E}\left[  f^{\prime
}\left(  X\right)  \right]  -\mathbf{E}\left[  Xf\left(  X\right)  \right]  $;

\item use their observation to transform this expression into $\mathbf{E}%
\left[  f^{\prime}\left(  X\right)  \left(  1-G\right)  \right]  $;

\item use the boundedness and decay properties of $f^{\prime}$ (these are
classically known from Stein's equation) to exploit the proximity of $G$ to
$1$.
\end{enumerate}

As we said, this strategy of relating Stein's method and the Malliavin
calculus is particularly useful for analyzing problems in stochastic analysis.
In addition to their study of convergence in Wiener chaos in \cite{NP}, which
they followed up with sharper results in \cite{NP2}, Nourdin and Peccati have
implemented several other applications including: the study of cummulants on
Wiener chaos \cite{NPcum}, of fluctuations of Hermitian random matrices
\cite{NPmtx}, and, with other authors, other results about the structure of
inequalities and convergences on Wiener space, such as \cite{BNP},
\cite{NPRe}, \cite{NPRe2}, \cite{NPRev}. In \cite{NV}, it was pointed out that
if $\rho$ denotes the density of $X$, then the function
\begin{equation}
g(z):=\rho^{-1}(z)\int_{z}^{\infty}y\rho(y)\,dy, \label{geeintro}%
\end{equation}
which was originally defined by Stein in \cite{Stein}, can be represented as%
\[
g\left(  z\right)  =\mathbf{E}[G|X=z],
\]
resulting in a convenient formula for the density $\rho$, which was then
exploited to provide new Gaussian lower bound results for certain stochastic
models, in \cite{NV} for Gaussian fields, and subsequently in \cite{Viens} for
polymer models in Gaussian and non-Gaussian environments, in \cite{nq} for
stochastic heat equations, in \cite{BNP} for statistical inference for
long-memory stochastic processes, and multivariate extensions of density
formulas in \cite{AMV}.

\subsection{Summary of our results}

Our specific motivation is drawn from the results in \cite{Viens} which make
assumptions on how $G$ compares to $1$ almost surely, and draw conclusions on
how the tail of $X$, i.e. $\mathbf{P}\left[  X>z\right]  $, compares to the
normal tail $\mathbf{P}\left[  Z>z\right]  $. By the above observations, these
types of almost-sure assumptions are equivalent to comparing the deterministic
function $g$ to the value $1$. For instance, one result in \cite{Viens} can be
summarized by saying that (under some additional regularity conditions) if
$G\geq1$ almost surely, i.e. if $g\left(  z\right)  \geq1$ everywhere, then
for some constant $c$ and large enough $z$, $\mathbf{P}\left[  X>z\right]
>c\mathbf{P}\left[  Z>z\right]  $. This result, and all the ones mentioned
above, concentrate on comparing laws to the standard normal law, which is done
by comparing $G$ to the constant $1$, as this constant is the
\textquotedblleft$G$\textquotedblright\ for the standard normal $Z$.\bigskip

In this paper, we find a framework which enables us to compare the law of $X$
to a wide range of laws. Instead of assuming that $g$ is comparable to $1$, we
only assume that it is comparable to a polynomial of degree less than or equal
to $2$. In \cite{Stein}, Stein had originally noticed that the set of all
distributions such that their $g$ is such a polynomial, is precisely the
so-called Pearson class of distributions. They encompass Gaussian, Gamma, and
Beta distributions, as well as the inverse-Gamma, and a number of continuous
distributions with only finitely many moments, with prescribed power tail
behavior. This means that one can hope to give precise criteria based on $g$,
or via Malliavin calculus based on $G$, to guarantee upper and/or lower bounds
on the tail $\mathbf{P}\left[  X>z\right]  $, with various Gaussian,
exponential, or power-type behaviors. We achieve such results in this paper.

Specifically, our first set of results is in the following general framework.
Let $Z$ be a reference random variable supported on $\left(  a,b\right)  $
where $-\infty\leq a<b\leq+\infty$, with a density $\rho_{\ast}$ which is
continuous on $\mathbf{R}$ and differentiable on $(a,b)$. The function $g$
corresponding to $\rho_{\ast}$ is given as in (\ref{geeintro}), and we denote
it by $g_{\ast}$ (the subscripts $_{\ast}$ indicate that these are relative to
our reference r.v.):%
\begin{equation}
g_{\ast}\left(  z\right)  =\frac{\int_{z}^{\infty}y\rho_{\ast}(y)\,dy}%
{\rho_{\ast}(z)}\mathbf{1}_{(a,b)}(z). \label{geestar}%
\end{equation}
We also use the notation%
\[
\Phi_{\ast}\left(  z\right)  =\mathbf{P}\left[  Z>z\right]
\]
for our reference tail. Throughout this article, for notational convenience,
we assume that $Z$ is centered (except when specifically stated otherwise in
Section \ref{PearsonList} in the Appendix). Let $X$ be
Malliavin-differentiable, supported on $(a,b)$, with its $G:=\langle
DX;-DL^{-1}X\rangle$ as in (\ref{Gee}).

\begin{itemize}
\item (Theorem \ref{thm:TailLBound}) Under mild regularity and integrability
conditions on $Z$ and $X$, if $G\geq g_{\ast}\left(  X\right)  $ almost
surely, then for all $z<b$,%
\[
\mathbf{P}[X>z]\geq\Phi_{\ast}(z)-\frac{1}{Q(z)}\int_{z}^{b}(2y-z)\mathbf{P}%
[X>y]\,dy,
\]
where%
\begin{equation}
Q\left(  z\right)  :=z^{2}-zg_{\ast}\left(  z\right)  +g_{\ast}\left(
z\right)  ; \label{Q}%
\end{equation}
typically $Q$ is of order $z^{2}$ for large $z$.

\item (Theorem \ref{thm:TailUBound}) Under mild regularity and integrability
conditions on $Z$ and $X$, if $G\leq g_{\ast}\left(  X\right)  $ almost
surely, then for some constant $c$ and all large enough $z<b$,%
\[
\mathbf{P}[X>z]\leq c\Phi_{\ast}(z).
\]

\end{itemize}

These results are generalizations of the work in \cite{Viens}, where only the
standard normal $Z$ was considered. They can be rephrased by referring to $g$
as in (\ref{geeintro}), which coincides with $g\left(  z\right)
=\mathbf{E}\left[  G|X=z\right]  $, rather than $G$; this can be useful to
apply the theorems in contexts where the definition of $X$ as a member of a
Wiener space is less explicit than the information one might have directly
about $g$. We have found, however, that the Malliavin-calculus interpretation
makes for efficient proofs of the above theorems.

The main application of these general theorems are to the Pearson class: $Z$
such that its $g_{\ast}$ is of the form $g_{\ast}\left(  z\right)  =\alpha
z^{2}+\beta z+\gamma$ in the support of $Z$. Assume $b=+\infty$, i.e. the
support of $Z$ is $(a,+\infty)$. Assume $\mathbf{E}\left[  |Z|^{3}\right]
<\infty$ (which is equivalent to $\alpha<1/2$). Then the lower bound above can
be made completely explicit, as can the constant $c$ in the upper bound.

\begin{itemize}
\item (Corollary \ref{cor:TailLBoundP}) Under mild regularity and
integrability conditions on $X$ [including assuming that there exists $c>2$
such that $g\left(  z\right)  \leq z^{2}/c$ for large $z$], if $G\geq g_{\ast
}\left(  X\right)  $ almost surely, then for any $c^{\prime}<1/\left(
1+2(1-\alpha)(c-2)\right)  $ and all $z$ large enough,%
\[
\mathbf{P}[X>z]\geq c^{\prime}\Phi_{\ast}(z).
\]

\item (Corollary \ref{thm:TailUBoundP2}) Under mild regularity and
integrability conditions on $X$, if $G\leq g_{\ast}\left(  X\right)  $ almost
surely, then for any $c>(1-\alpha)/(1-2\alpha)$, and all $z$ large enough,%
\[
\mathbf{P}[X>z]\leq c\Phi_{\ast}(z).
\]

\end{itemize}

The results above can be used conjointly with asymptotically sharp conclusions
when upper and lower bound assumptions on $G$ are true simultaneously. For
instance, we have the following, phrased using $g$'s instead of $G$'s.

\begin{itemize}
\item (Corollary \ref{cor:TailX2}, point 2) On the support $(a,+\infty)$, let
$g_{\ast}\left(  z\right)  =\alpha z^{2}+\beta z+\gamma$ and let $\bar
{g}_{\ast}\left(  z\right)  =\bar{\alpha}z^{2}+\bar{\beta}z+\bar{\gamma}$ with
non-zero $\alpha$ and $\bar{\alpha}$. If for the Malliavin-differentiable $X$
and its corresponding $g$, we have for all $z>a$, $g_{\ast}\left(  z\right)
\leq g\left(  z\right)  \leq\bar{g}_{\ast}\left(  z\right)  $, then there are
constants $c$ and $\bar{c}$ such that for large $z$,%
\[
cz^{-1-1/\alpha}\leq\mathbf{P}[X>z]\leq\bar{c}z^{-1-1/\bar{\alpha}}.
\]

\item (see Corollary \ref{cor:TailX1}) A similar result holds when
$\alpha=\bar{\alpha}=0$, in which $\mathbf{P}[X>z]$ compares to the Gamma-type
tail $z^{-1-\gamma/\beta^{2}}\exp\left(  -z/\beta\right)  $.\bigskip
\end{itemize}

The strategy used to prove these results is an analytic one, following the
initial method of Nourdin and Peccati, this time using the Stein equation
relative to the function $g_{\ast}$ defined in (\ref{geestar}) for a general
reference r.v. $Z$:%
\[
g_{\ast}\left(  x\right)  f^{\prime}\left(  x\right)  -xf\left(  x\right)
=h\left(  x\right)  -\mathbf{E}\left[  h\left(  Z\right)  \right]  .
\]
Our mathematical techniques are based on a careful analysis of the properties
of $g_{\ast}$, its relation to the function $Q$ defined in (\ref{Q}), and what
consequences can be derived for the solutions of Stein's equation. The basic
general theorems' proofs use a structure similar to that employed in
\cite{Viens}. The applications to the Pearson class rely heavily on explicit
computations tailored to this case, which are facilitated via the
identification of $Q$ as a useful way to express these computations.

This article is structured as follows. Section 2 gives an overview of Stein's
equations, and derives some fine properties of their solutions by referring to
the function $Q$. These will be crucial in the proofs of our general upper and
lower bound results, which are presented in Section 3 after an overview of the
tools of Malliavin calculus which are needed in this article. Applications to
comparisons with Pearson distributions, with a particular emphasis on tail
behavior, including asymptotic results, are in Section 4. Section 5 is an
Appendix containing the proofs of some technical lemmas and some details on
Pearson distributions.

This article is dedicated to the memory of Professor Paul Malliavin.

\section{The Stein equation}

\subsection{Background and classical results}

\noindent\textbf{Characterization of the law of }$Z$. As before, let $Z$ be
centered with a differentiable density on its support $(a,b)$, and let
$g_{\ast}$ be defined as in (\ref{geestar}). Nourdin and Peccati (Proposition
6.4 in \cite{NP}) collected the following results concerning this equation. If
$f$ is a function that is continuous and piecewise continuously
differentiable, and if $\mathbf{E}[|f^{\prime}(Z)|g_{\ast}(Z)]<\infty$, Stein
(Lemma 1, p. 59 in \cite{Stein}) proved that
\[
\mathbf{E}[g_{\ast}(Z)f^{\prime}(Z)-Zf(Z)]=0
\]
(compare this with (\ref{stein_normal}) for the special case $Z\thicksim
\mathcal{N}(0,1)$). Conversely, assume that%
\begin{equation}
\int_{0}^{b}\frac{z}{g_{\ast}(z)}\,dz=\infty\qquad\mathrm{and}\qquad\int
_{a}^{0}\frac{z}{g_{\ast}(z)}\,dz=-\infty. \label{tau_cond}%
\end{equation}
If a random variable $X$ has a density, and for any differentiable function
$f$ such that $x\mapsto|g_{\ast}(x)f^{\prime}(x)|+|xf(x)|$ is bounded,%
\begin{equation}
\mathbf{E}[g_{\ast}(X)f^{\prime}(X)-Xf(X)]=0 \label{tau_characterize}%
\end{equation}
then $X$ and $Z$ have the same law. In other words, under certain conditions,
(\ref{tau_characterize}) can be used to characterize the law of a centered
random variable $X$ as being equal to that of $Z$.\bigskip

\noindent\textbf{Stein's equation, general case; distances between
distributions. }If $h$ is a fixed bounded piecewise continuous function such
that $\mathbf{E}[|h(Z)|]<\infty$, the corresponding \textit{Stein equation}
for $Z$ is the ordinary differential equation in $f$ defined by%
\begin{equation}
h(x)-\mathbf{E}[h(Z)]=g_{\ast}(x)f^{\prime}(x)-xf(x)\text{.}\label{stein}%
\end{equation}
The utility of such an equation is apparent when we evaluate the functions at
$X$ and take expectations:%
\begin{equation}
\mathbf{E}[h(X)]-\mathbf{E}[h(Z)]=\mathbf{E}[g_{\ast}(X)f^{\prime
}(X)-Xf(X)]\text{.}\label{stein_exp}%
\end{equation}
The idea is that if the law of $X$ is \textquotedblleft
close\textquotedblright\ to the law of $Z$, then the right side of
(\ref{stein_exp}) would be close to $0$. Conversely, if the test function $h$
can be chosen from specific classes of functions so that the left side of
(\ref{stein_exp}) denotes a particular notion of distance between $X$ and $Z$,
the closeness of the right-hand side of (\ref{stein_exp}) to zero, in some
uniform sense in the $f$'s satisfying Stein's equation (\ref{stein}) for all
the $h$'s in that specific class of test functions, will imply that the laws
of $X$ and $Z$ are close in the corresponding distance. For this purpose, it
is typically crucial to establish boundedness properties of $f$ and
$f^{\prime}$ which are uniform over the class of test functions being considered.

For example, if $\mathcal{H=}\{h:||h||_{L}+||h||_{\infty}\leq1\}$ where
$||\cdot||_{L}$ is the Lipschitz seminorm, then the Fortet-Mourier distance
$d_{FM}(X,Z)$ between $X$ and $Z$ is defined as
\[
d_{FM}(X,Z)=\sup_{h\in\mathcal{H}}|\mathbf{E}[h(X)]-\mathbf{E}[h(Z)]|.
\]
This distance metrizes convergence in distribution, so by using properties of
the solution $f$ of the Stein equation (\ref{stein_exp}) for $h\in\mathcal{H}%
$, we can draw conclusions on the convergence in distribution of a sequence
$\{X_{n}\}$ to $Z$. See \cite{NP} and \cite{Dudley}\ for details and other
notions of distance between random variables.\bigskip

\noindent\textbf{Solution of Stein's equation. }Stein (Lemma 4, p. 62 in
\cite{Stein}) proved that if (\ref{tau_cond})\ is satisfied, then his equation
(\ref{stein}) has a unique solution $f$ which is bounded and continuous on
$(a,b)$. If $x\notin(a,b)$, then%
\begin{equation}
f(x)=-\frac{h(x)-\mathbf{E}\left[  h(Z)\right]  }{x} \label{stein_sol0}%
\end{equation}
while if $x\in(a,b)$,%
\begin{equation}
f(x)=\int_{a}^{x}\left(  h(y)-\mathbf{E}\left[  h(Z)\right]  \right)
\frac{e^{\int_{y}^{x}\frac{z\,dz}{g_{\ast}(z)}}}{g_{\ast}(y)}\,dy.
\label{stein_sol1}%
\end{equation}

\subsection{Properties of solutions of Stein's equations}

We assume throughout that $\rho_{\ast}$ is differentiable on $(a,b)$ and
continuous on $\mathbf{R}$ (for which it is necessary that $\rho_{\ast}$ be
null on $\mathbf{R}-(a,b)\ $). Consequently, $g_{\ast}$ is differentiable and
continuous on $(a,b)$. The next lemma records some elementary properties of
$g_{\ast}$, such as its positivity and its behavior near $a$ and $b$. Those
facts which are not evident are established in the Appendix. All are useful in
facilitating the proofs of other lemmas presented in this section, which are
key to our article.

\begin{lemma}
\label{lemgstar}Let $Z$ be centered and continuous, with a density $\rho
_{\ast}$ that is continuous on $\mathbf{R}$ and differentiable on its support
$(a,b)$, with $a$ and $b$ possibly infinite.

\begin{enumerate}
\item $g_{\ast}\left(  x\right)  >0$ if and only if $x\in(a,b);$

\item $g_{\ast}$ is differentiable on $(a,b)$ and we have $[g_{\ast}%
(x)\rho_{\ast}(x)]^{\prime}=-x\rho_{\ast}(x)$ therein;

\item $\lim\limits_{x\rightarrow a}g_{\ast}(x)\rho_{\ast}(x)=\lim
\limits_{x\rightarrow b}g_{\ast}(x)\rho_{\ast}(x)=0.$
\end{enumerate}
\end{lemma}

A different expression for the solution $f$ of Stein's equation (\ref{stein})
than the one given in (\ref{stein_sol0}, \ref{stein_sol1}), which will be more
convenient for our purposes, such as computing $f^{\prime}$ in the support of
$Z$, was given by Schoutens \cite{Schoutens} as stated in the next lemma.

\begin{lemma}
\label{lemschou}For all $x\in(a,b)$,%
\begin{equation}
f(x)=\frac{1}{g_{\ast}(x)\rho_{\ast}(x)}\int_{a}^{x}\left(  h(y)-\mathbf{E}%
\left[  h(Z)\right]  \right)  \rho_{\ast}(y)\,dy. \label{stein_sol2}%
\end{equation}
If $x\notin\left[  a,b\right]  $, differentiating (\ref{stein_sol0}) gives
\begin{equation}
f^{\prime}(x)=\frac{-xh^{\prime}(x)+h(x)-\mathbf{E}\left[  h(Z)\right]
}{x^{2}} \label{deriv2}%
\end{equation}
while if $x\in(a,b)$, differentiating (\ref{stein_sol2}) gives%
\begin{equation}
f^{\prime}(x)=\frac{x}{[g_{\ast}(x)]^{2}\rho_{\ast}(x)}\int_{a}^{x}\left(
h(y)-\mathbf{E}[h(Z)]\right)  \rho_{\ast}(y)\,dy+\frac{h(x)-\mathbf{E}%
[h(Z)]}{g_{\ast}(x)}. \label{deriv1}%
\end{equation}

\end{lemma}

The proof of this lemma (provided in the Appendix for completeness) also gives
us the next one.

\begin{lemma}
\label{lemchar}Under our assumption of differentiability on $(a,b)$ of
$\rho_{\ast}$ and hence of $g_{\ast}$, Stein's condition (\ref{tau_cond}) on
$g_{\ast}$ is satisfied.
\end{lemma}

In Stein's equation (\ref{stein}), the test function $h=1_{(-\infty,z]}$ lends
itself to useful tail probability results since $\mathbf{E}[h(Z)]=\mathbf{P}%
[Z\leq z]$. From this point on, we will assume that $h=1_{(-\infty,z]}$ with
fixed $z>0$, and that $f$ is the corresponding solution of Stein's equation
(we could denote the parametric dependence of $f$ on $z$ by $f_{z}$, but
choose to omit the subscript to avoid overburdening the notation).

As opposed to the previous lemmas, the next two results, while still
elementary in nature, appear to be new, and their proofs, which require some
novel ideas of possibly independent interest, have been kept in the main body
of this paper, rather than having them relegated to the Appendix. We begin
with an analysis of the sign of $f^{\prime}$, which will be crucial to prove
our main general theorems.

\begin{lemma}
\label{lem:signfprime} Suppose $0<z<b$. If $x\leq z$, then $f^{\prime}%
(x)\geq0$. If $x>z$, then $f^{\prime}(x)\leq0$.
\end{lemma}

\begin{proof}
The result follows easily from (\ref{deriv2}) when $x\notin\left[  a,b\right]
$: if $x<a$, then $f^{\prime}(x)=\left(  1-\mathbf{E}[h(Z)]\right)  /x^{2}%
\geq0$, while if $x>b$, then $f^{\prime}(x)=-\mathbf{E}[h(Z)]/x^{2}\leq0$. So
now we can assume that $x\in(a,b)$. We will use the expression for the
derivative $f^{\prime}$ given in (\ref{deriv1}).

Suppose $a<x\leq z$. Then $h(x)=1$ and for any $y\leq x$, $h(y)=1$ so
\begin{align*}
f^{\prime}(x) &  =\frac{x}{[g_{\ast}(x)]^{2}\rho_{\ast}(x)}\int_{a}^{x}\left(
1-\mathbf{E}[h(Z)]\right)  \rho_{\ast}(y)\,dy+\frac{1-\mathbf{E}%
[h(Z)]}{g_{\ast}(x)}\\
&  =\frac{1-\mathbf{E}[h(Z)]}{[g_{\ast}(x)]^{2}\rho_{\ast}(x)}\left(
x\int_{a}^{x}\rho_{\ast}(y)\,dy+g_{\ast}(x)\rho_{\ast}(x)\right)
\end{align*}
Clearly, $f^{\prime}(x)\geq0$ if $x\geq0$. Now define%
\[
n_{1}(x):=\int_{a}^{x}\rho_{\ast}(y)\,dy+\frac{g_{\ast}(x)\rho_{\ast}(x)}{x}.
\]
We will show that $xn_{1}(x)\geq0$ when $x<0$. Since
\begin{align*}
n_{1}^{\prime}(x) &  =\rho_{\ast}(x)+\frac{x[g_{\ast}(x)\rho_{\ast
}(x)]^{\prime}-g_{\ast}(x)\rho_{\ast}(x)}{x^{2}}\\
&  =\rho_{\ast}(x)+\frac{-x^{2}\rho_{\ast}(x)-g_{\ast}(x)\rho_{\ast}(x)}%
{x^{2}}=-\frac{g_{\ast}(x)\rho_{\ast}(x)}{x^{2}}\leq0
\end{align*}
then $n_{1}$ is non-increasing on $\left(  a,0\right)  $ which means that
whenever $a<x<0$, $n_{1}(x)\leq\lim\limits_{x\rightarrow a}n_{1}%
(x)=\lim\limits_{x\rightarrow a}\frac{g_{\ast}(x)p(x)}{x}=0$ since
$\lim\limits_{x\rightarrow a}g_{\ast}(x)\rho_{\ast}(x)=0$. Therefore,
$xn_{1}(x)\geq0$ for $x<0$. This completes the proof that $f^{\prime}(x)\geq0$
whenever $x\leq z$.

Finally, suppose that $z<x<b$ so $h(x)=0$. Since $\mathbf{E}[h(Z)]=\mathbf{P}%
[Z\leq z]=\int_{a}^{z}\rho_{\ast}(y)\,dy$ $,$%
\begin{align*}
f^{\prime}(x)  &  =\frac{x}{[g_{\ast}(x)]^{2}\rho_{\ast}(x)}\int_{a}%
^{x}h(y)\rho_{\ast}(y)\,dy-\frac{x\mathbf{E}[h(Z)]}{[g_{\ast}(x)]^{2}%
\rho_{\ast}(x)}\int_{a}^{x}\rho_{\ast}(y)\,dy-\frac{\mathbf{E}[h(Z)]}{g_{\ast
}(x)}\\
&  =\frac{x}{[g_{\ast}(x)]^{2}\rho_{\ast}(x)}\int_{a}^{z}\rho_{\ast
}(y)\,dy-\frac{x\mathbf{E}[h(Z)]}{[g_{\ast}(x)]^{2}\rho_{\ast}(x)}\int_{a}%
^{x}\rho_{\ast}(y)\,dy-\frac{\mathbf{E}[h(Z)]}{g_{\ast}(x)}\\
&  =\frac{x}{[g_{\ast}(x)]^{2}\rho_{\ast}(x)}\mathbf{E}[h(Z)]-\frac
{x\mathbf{E}[h(Z)]}{[g_{\ast}(x)]^{2}\rho_{\ast}(x)}\int_{a}^{x}\rho_{\ast
}(y)\,dy-\frac{\mathbf{E}[h(Z)]}{g_{\ast}(x)}\\
&  =\frac{\mathbf{E}[h(Z)]}{[g_{\ast}(x)]^{2}\rho_{\ast}(x)}\left(
x-x\int_{a}^{x}\rho_{\ast}(y)\,dy-g_{\ast}(x)\rho_{\ast}(x)\right) \\
&  =\frac{\mathbf{E}[h(Z)]}{[g_{\ast}(x)]^{2}\rho_{\ast}(x)}\cdot xn_{2}(x)
\end{align*}
where
\[
n_{2}(x):=1-\int_{a}^{x}\rho_{\ast}(y)\,dy-\frac{g_{\ast}(x)\rho_{\ast}(x)}%
{x}=1-n_{1}(x).
\]
It is enough to show that $n_{2}(x)\leq0$ since $x>z\geq0$. Since
$n_{2}^{\prime}(x)=-n_{1}^{\prime}(x)\geq0$, then $n_{2}(x)\leq\lim
\limits_{x\rightarrow b}n_{2}(x)=1-\lim\limits_{x\rightarrow b}\int_{a}%
^{x}\rho_{\ast}(y)\,dy-\lim\limits_{x\rightarrow b}\frac{g_{\ast}(x)\rho
_{\ast}(x)}{x}=0$ because $\lim\limits_{x\rightarrow b}g_{\ast}(x)\rho_{\ast
}(x)=0$. Therefore, $f^{\prime}(x)\leq0$ if $x>z$, finishing the proof of the lemma.
\end{proof}

\bigskip

As alluded to in the previous subsection, of crucial importance in the use of
Stein's method, is a quantitatively explicit boundedness result on the
derivative of the solution to Stein's equation. We take this up in the next lemma.

\begin{lemma}
\label{lem:fprimebounded}Recall the function%
\[
Q(x):=x^{2}-xg_{\ast}^{\prime}(x)+g_{\ast}(x)
\]
defined in (\ref{Q}), for all $x\in\mathbf{R}$ except possibly at $a$ and $b$.
Assume that $g_{\ast}^{\prime\prime}(x)<2$ for all $x$ and that $\frac
{x-g_{\ast}^{\prime}(x)}{Q(x)}$ tends to a finite limit as $x\rightarrow a$
and as $x\rightarrow b$. Suppose $0<z<b$. Then $f^{\prime}(x)$ is bounded. In
particular, if $a<x\leq z$,%
\begin{equation}
0\leq f^{\prime}(x)\leq\frac{z}{[g_{\ast}(z)]^{2}\rho_{\ast}(z)}+\frac
{1}{Q(0)}<\infty, \label{f'bound1}%
\end{equation}
while if $b>x>z$,%
\begin{equation}
-\infty<-\frac{1}{Q(z)}\leq f^{\prime}(x)\leq0. \label{f'bound2}%
\end{equation}

\end{lemma}

To prove this lemma, we need two auxiliary results. The first one introduces
and studies the function $Q$ which we already encountered in the introduction,
and which will help us state and prove our results in an efficient way. The
second one shows the relation between $Q$, $g_{\ast}$, and the tail
$\Phi_{\ast}$ of $Z$, under conditions which will be easily verified later on
in the Pearson case.

\begin{lemma}
\label{lemQ}~

\begin{enumerate}
\item If $x\notin(a,b)$, then $Q(x)=x^{2}>0$.

\item If $g_{\ast}$ is twice differentiable in $(a,b)$ (for example, when
$\rho_{\ast}$ is twice differentiable), then $Q^{\prime}(x)=x\left(
2-g_{\ast}^{\prime\prime}(x)\right)  $.

\item If moreover $g_{\ast}^{\prime\prime}(x)<2$ everywhere in $(a,b)$, a
reasonable assumption as we shall see later when $Z$ is a Pearson random
variable, then $\min{}_{(a,b)}Q=Q\left(  0\right)  $ so that $Q(x)\geq
Q(0)=g_{\ast}(0)>0$.
\end{enumerate}
\end{lemma}

\begin{lemma}
\label{lem:phibound} With the assumptions on $g_{\ast}$ and $Q$ as in Lemma
\ref{lem:fprimebounded}, then for all $x$,
\begin{equation}
\frac{\max\left(  x-g_{\ast}^{\prime}(x),0\right)  }{Q(x)}g_{\ast}%
(x)\rho_{\ast}(x)\leq\Phi_{\ast}(x) \label{tail1a}%
\end{equation}
and\thinspace%
\begin{equation}
\frac{\max\left(  g_{\ast}^{\prime}(x)-x,0\right)  }{Q(x)}g_{\ast}%
(x)\rho_{\ast}(x)\leq1-\Phi_{\ast}(x). \label{tail1b}%
\end{equation}
Moreover for $0<x<b$, we have%
\begin{equation}
\Phi_{\ast}(x)\leq\frac{1}{x}\cdot g_{\ast}(x)\rho_{\ast}(x) \label{tail2}%
\end{equation}
while if $a<x<0$, then%
\begin{equation}
1-\Phi_{\ast}(x)\leq\frac{1}{-x}\cdot g_{\ast}(x)\rho_{\ast}(x). \label{tail3}%
\end{equation}

\end{lemma}

\begin{proof}
[Proof of Lemma \ref{lem:fprimebounded}]If $x<a$ with $a>-\infty$, then
$f^{\prime}(x)=\left(  1-\mathbf{E}[h(Z)]\right)  /x^{2}\leq\left(
1-\mathbf{E}[h(Z)]\right)  /a^{2}$. If $x>b$ with $b<\infty$, then $f^{\prime
}(x)=-\mathbf{E}[h(Z)]/x^{2}\geq-\mathbf{E}[h(Z)]/b^{2}$. So now we only need
to assume that $x\in(a,b)$.

Suppose $a<x\leq z$. Use $f^{\prime}(x)\geq0$ given in (\ref{deriv1}):%
\begin{align*}
f^{\prime}(x)  &  =\frac{1-\mathbf{E}[h(Z)]}{[g_{\ast}(x)]^{2}\rho_{\ast}%
(x)}\left(  x\int_{a}^{x}\rho_{\ast}(y)\,dy+g_{\ast}(x)\rho_{\ast}(x)\right)
\\
&  \leq\frac{x}{[g_{\ast}(x)]^{2}\rho_{\ast}(x)}\left(  1-\Phi_{\ast
}(x)\right)  +\frac{1}{g_{\ast}(x)}%
\end{align*}
When $x\geq0$, we can rewrite the upper bound as:
\[
f^{\prime}(x)\leq r(x)+\frac{1}{g_{\ast}(x)}\left[  1-\frac{x\Phi_{\ast}%
(x)}{g_{\ast}(x)\rho_{\ast}(x)}\right]
\]
where
\[
r(x)=\frac{x}{[g_{\ast}(x)]^{2}\rho_{\ast}(x)}=\frac{x}{g_{\ast}(x)[g_{\ast
}(x)\rho_{\ast}(x)]}.
\]
We can bound $r(x)$ above since
\begin{align*}
r^{\prime}(x)  &  =\frac{[g_{\ast}(x)]^{2}\rho_{\ast}(x)-x\left[  g_{\ast
}(x)\left(  -x\rho_{\ast}(x)\right)  +g_{\ast}(x)\rho_{\ast}(x)g_{\ast
}^{\prime}(x)\right]  }{[g_{\ast}(x)]^{4}\left[  \rho_{\ast}(x)\right]  ^{2}%
}\\
&  =\frac{g_{\ast}(x)+x^{2}-xg_{\ast}^{\prime}(x)}{[g_{\ast}(x)]^{3}\rho
_{\ast}(x)}=\frac{Q(x)}{[g_{\ast}(x)]^{3}\rho_{\ast}(x)}>0
\end{align*}
so $r(x)\leq r(z)$. To bound $\left[  1-x\Phi_{\ast}(x)/\left(  g_{\ast
}(x)\rho_{\ast}(x)\right)  \right]  /g_{\ast}(x)$, use (\ref{tail1a}) of Lemma
\ref{lem:phibound}:
\begin{align*}
\frac{1}{g_{\ast}(x)}\left[  1-\frac{x\Phi_{\ast}(x)}{g_{\ast}(x)\rho_{\ast
}(x)}\right]   &  \leq\frac{1}{g_{\ast}(x)}\left[  1-\frac{x}{g_{\ast}%
(x)\rho_{\ast}(x)}\cdot\frac{x-g_{\ast}^{\prime}(x)}{Q(x)}g_{\ast}%
(x)\rho_{\ast}(x)\right] \\
&  =\frac{1}{g_{\ast}(x)}\left[  1-\frac{x^{2}-xg_{\ast}^{\prime}(x)}%
{Q(x)}\right] \\
&  =\frac{1}{g_{\ast}(x)}\cdot\frac{g_{\ast}(x)}{Q(x)}=\frac{1}{Q(x)}\leq
\frac{1}{Q(0)}.
\end{align*}
Therefore,
\[
f^{\prime}(x)\leq\frac{z}{[g_{\ast}(z)]^{2}\rho_{\ast}(z)}+\frac{1}{Q(0)}.
\]

When $x<0$, we use (\ref{tail1b}) of Lemma \ref{lem:phibound}:%
\begin{align*}
f^{\prime}(x)  &  \leq\frac{x}{[g_{\ast}(x)]^{2}\rho_{\ast}(x)}\left(
1-\Phi_{\ast}(x)\right)  +\frac{1}{g_{\ast}(x)}\\
&  \leq\frac{x}{[g_{\ast}(x)]^{2}\rho_{\ast}(x)}\cdot\frac{g_{\ast}^{\prime
}(x)-x}{Q(x)}g_{\ast}(x)\rho_{\ast}(x)+\frac{1}{g_{\ast}(x)}\\
&  =\frac{1}{g_{\ast}(x)}\cdot\frac{xg_{\ast}^{\prime}(x)-x^{2}}{Q(x)}%
+\frac{1}{g_{\ast}(x)}\\
&  =\frac{1}{g_{\ast}(x)}\left[  \frac{xg_{\ast}^{\prime}(x)-x^{2}+Q(x)}%
{Q(x)}\right]  =\frac{1}{Q(x)}\leq\frac{1}{Q(0)}\\
&  \leq\frac{z}{[g_{\ast}(z)]^{2}\rho_{\ast}(z)}+\frac{1}{Q(0)}.
\end{align*}

Now we prove (\ref{f'bound2}) and so suppose $x>z>0$. From the proof of Lemma
\ref{lem:signfprime},%
\begin{align*}
f^{\prime}(x)  &  =\frac{\mathbf{E}[h(Z)]}{[g_{\ast}(x)]^{2}\rho_{\ast}%
(x)}\left(  x-x\int_{a}^{x}\rho_{\ast}(y)\,dy-g_{\ast}(x)\rho_{\ast}(x)\right)
\\
&  =\frac{\mathbf{E}[h(Z)]}{[g_{\ast}(x)]^{2}\rho_{\ast}(x)}\left(
x\Phi_{\ast}(x)-g_{\ast}(x)\rho_{\ast}(x)\right)  .
\end{align*}
We conclude by again using (\ref{tail1a}) of Lemma \ref{lem:phibound}, to get%
\begin{align*}
-f^{\prime}(x)  &  \leq\frac{1}{[g_{\ast}(x)]^{2}\rho_{\ast}(x)}\left(
g_{\ast}(x)\rho_{\ast}(x)-x\Phi_{\ast}(x)\right) \\
&  \leq\frac{1}{[g_{\ast}(x)]^{2}\rho_{\ast}(x)}\left(  g_{\ast}(x)\rho_{\ast
}(x)-x\cdot\frac{x-g_{\ast}^{\prime}(x)}{Q(x)}\cdot g_{\ast}(x)\rho_{\ast
}(x)\right) \\
&  =\frac{1}{g_{\ast}(x)}\left(  1-\frac{x^{2}-xg_{\ast}^{\prime}(x)}%
{Q(x)}\right)  =\frac{1}{g_{\ast}(x)}\cdot\frac{g_{\ast}(x)}{Q(x)}\\
&  =\frac{1}{Q(x)}\leq\frac{1}{Q(z)}.
\end{align*}

\end{proof}

\bigskip

\begin{remark}
Since $z>0$, and $Q\left(  z\right)  >Q\left(  0\right)  $, Lemma
\ref{lem:fprimebounded} implies the following convenient single bound for any
fixed $z>0$, uniform for all $x\in(a,b)$:
\[
\mathbf{|}f^{\prime}(x)\mathbf{|}\leq\frac{z}{[g_{\ast}(z)]^{2}\rho_{\ast}%
(z)}+\frac{1}{Q(0)}.
\]
Occasionally, this will be sufficient for some of our purposes. The more
precise bounds in Lemma \ref{lem:fprimebounded} will also be needed, however.
\end{remark}

\section{Main results}

In order to exploit the boundedness of $f^{\prime}$, we adopt the technique
pioneered by Nourdin and Peccati, to rewrite expressions of the form
$\mathbf{E}[Xm(X)]$ where $m$ is a function, using the Malliavin calculus. For
ease of reference, we include here the requisite Malliavin calculus
constructs. Full details can be found in \cite{Nualart}; also see
\cite[Section 2]{Viens} for an exhaustive summary.

\subsection{Elements of Malliavin calculus\label{MALL}}

We assume our random variable $X$ is measurable with respect to an isonormal
Gaussian process $W$, associated with its canonical separable Hilbert space
$H$. For illustrative purposes, one may further assume, as we now do, that $W$
is the standard white-noise corresponding to $H=L^{2}\left(  [0,1]\right)  $,
which is constructed using a standard Brownian motion on $[0,1]$, also denoted
by $W$, endowed with its usual probability space $\left(  \Omega
,\mathcal{F},\mathbf{P}\right)  $. This means that the white noise $W$ is
defined by $W\left(  f\right)  =\int_{0}^{1}f\left(  s\right)  dW\left(
s\right)  $ for any $f\in H$, where the stochastic integral is the Wiener
integral of $f$ with respect to the Wiener process $W$. If we denote
$I_{0}\left(  f\right)  =f$ for any non-random constant $f$, then for any
integer $n\geq1$ and any symmetric function $f\in H^{n}$, we let%
\[
I_{n}\left(  f\right)  :=n!\int_{0}^{1}\int_{0}^{s_{1}}\cdots\int_{0}%
^{s_{n-1}}f\left(  s_{1},s_{2},\cdots,s_{n}\right)  dW\left(  s_{n}\right)
\cdots dW\left(  s_{2}\right)  dW\left(  s_{1}\right)  ,
\]
where this integral is an iteration of $n$ It\^{o} integrals. It is called the
$n$th multiple Wiener integral of $f$ w.r.t. $W$, and the set $\mathcal{H}%
_{n}:=\left\{  I_{n}\left(  f\right)  :f\in H^{n}\right\}  $ is the $n$th
Wiener chaos of $W$. Note that $I_{1}\left(  f\right)  =W\left(  f\right)  $,
and that $\mathbf{E}\left[  I_{n}\left(  f\right)  \right]  =0$ for all
$n\geq1$. Again, see \cite[Section 1.2]{Nualart} for the general definition of
$I_{n}$ and $\mathcal{H}_{n}$ when $W$ is a more general isonormal Gaussian
process. The main representation theorem of the analysis on Wiener space is
that $L^{2}\left(  \Omega,\mathcal{F},\mathbf{P}\right)  $ is the direct sum
of all the Wiener chaoses. In other words, $X\in L^{2}\left(  \Omega
,\mathcal{F},\mathbf{P}\right)  $ if and only if there exists a sequence of
non-random symmetric functions $f_{n}\in H^{n}$ with $\sum_{n=0}^{\infty
}\left\Vert f_{n}\right\Vert _{H^{n}}^{2}<\infty$ such that $X=\sum
_{n=0}^{\infty}I_{n}\left(  f_{n}\right)  $. Note that $\mathbf{E}\left[
X\right]  =f_{0}$. Moreover, the terms in this so-called Wiener chaos
decomposition of $X$ are orthogonal in $L^{2}\left(  \Omega,\mathcal{F}%
,\mathbf{P}\right)  $, and we have the isometry property $\mathbf{E}\left[
X^{2}\right]  =\sum_{n=0}^{\infty}n!\left\Vert f_{n}\right\Vert _{H^{n}}^{2}$.
We are now in a position to define the Malliavin derivative $D$.

\begin{definition}
Let $\mathbf{D}^{1,2}$ be the subset of $L^{2}\left(  \Omega,\mathcal{F}%
,\mathbf{P}\right)  $ formed by those $X=\sum_{n=0}^{\infty}I_{n}\left(
f_{n}\right)  $ such that%
\[
\sum_{n=1}^{\infty}n\ n!\left\Vert f_{n}\right\Vert _{H^{n}}^{2}<\infty.
\]
The Malliavin derivative operator $D$ is defined from $\mathbf{D}^{1,2}$ to
$L^{2}\left(  \Omega\times\lbrack0,1]\right)  $ by $DX=0$ if $X=\mathbf{E}X$
is non-random, and otherwise, for all $r\in\lbrack0,1]$, by
\[
D_{r}X=\sum_{n=1}^{\infty}nI_{n-1}\left(  f_{n}\left(  r,\cdot\right)
\right)  .
\]

\end{definition}

This can be understood as a Fr\'{e}chet derivative of $X$ with respect to the
Wiener process $W$. If $X=W\left(  f\right)  $ then $DX=f$. Of note is the
chain-rule formula $D\left(  F\left(  X\right)  \right)  =F^{\prime}\left(
X\right)  DX$ for any differentiable $F$ with bounded derivative, and any
$X\in\mathbf{D}^{1,2}$.

\begin{definition}
The generator of the Ornstein-Uhlenbeck semigroup $L$ is defined as follows.
Let $X=\sum_{n=1}^{\infty}I_{n}\left(  f_{n}\right)  $ be a centered r.v. in
$L^{2}\left(  \Omega\right)  $. If $\sum_{n=1}^{\infty}n^{2}n!\left\vert
f_{n}\right\vert ^{2}<\infty$, then we define a new random variable $LX$ in
$L^{2}\left(  \Omega\right)  $ by $-LX=\sum_{n=1}^{\infty}nI_{n}\left(
f_{n}\right)  $. The pseudo-inverse of $L$ operating on centered r.v.'s in
$L^{2}\left(  \Omega\right)  $ is defined by the formula $-L^{-1}X=\sum
_{n=1}^{\infty}\frac{1}{n}I_{n}\left(  f_{n}\right)  .$ If $X$ is not
centered, we define its image by $L$ and $L^{-1}$ by applying them to
$X-\mathbf{E}X$.
\end{definition}

As explained in the introduction, for $X\in\mathbf{D}^{1,2}$, the random
variable $G:=\left\langle DX;-DL^{-1}X\right\rangle _{H}$ plays a crucial role
to understand how $X$'s law compares to that of our reference random variable
$Z$. The next lemma is the key to combining the solutions of Stein's equations
with the Malliavin calculus. Its use to prove our main theorems relies heavily
on the fact that these solutions have bounded derivatives.

\begin{lemma}
\label{lem:IntParts}(Theorem 3.1 in \cite{NP}, Lemma 3.5 in \cite{Viens}) Let
$X\in\mathbf{D}^{1,2}$ be a centered random variable with a density, and
$G=\langle DX;-DL^{-1}X\rangle_{H}$. For any deterministic, continuous and
piecewise differentiable function $m$ such that $m^{\prime}$ is bounded,%
\[
\mathbf{E}[Xm(X)]=\mathbf{E}[m^{\prime}(X)G].
\]

\end{lemma}

\subsection{General tail results}

The main theoretical results of this paper compare the tails of any two random
variables $X$ and $Z$, as we now state in the next two theorems. In terms of
their usage, $Z$ represents a reference random variable in these theorems;
this can be seen from the fact that we have a better control in the theorems'
assumption on the $g_{\ast}$ coming from $Z$ than on the law of $X$. Also,
will apply these theorems to a Pearson random variable $Z$ in the next
section, while there will be no restriction on $X\in\mathbf{D}^{1,2}$ beyond
the assumption of the theorems in the present section. we will see that all
assumptions on $Z$ in this section are satisfied when $Z$ is a Pearson random variable.

\begin{theorem}
\label{thm:TailLBound} Let $Z$ be a centered random variable with a twice
differentiable density over its support $(a,b)$. Let $g_{\ast}$ and $Q$ be
defined as in (\ref{geestar}) and (\ref{Q}), respectively.\ Suppose that
$g_{\ast}^{\prime\prime}(x)<2$, and $\frac{x-g_{\ast}^{\prime}(x)}{Q(x)}$ has
a finite limit as $x\rightarrow a$ and $x\rightarrow b$. Let $X\in
\mathbf{D}^{1,2}$ be a centered random variable with a density, and whose
support $(a,b_{X})$ contains $\left(  a,b\right)  $. Let $G$ be as in
(\ref{Gee}). If $G$ $\geq g_{\ast}\left(  X\right)  $ a.s., then for every
$z\in(0,b)$,
\begin{equation}
\mathbf{P}[X>z]\geq\Phi_{\ast}(z)-\frac{1}{Q(z)}\int_{z}^{b}(2x-z)\mathbf{P}%
[X>x]\,dx. \label{result1}%
\end{equation}

\end{theorem}

\begin{proof}
Taking expectations in \textit{Stein's equation} (\ref{stein}), i.e. referring
to (\ref{stein_exp}), we have%
\[
\mathbf{P}[X\leq z]-\mathbf{P}[Z\leq z]=\mathbf{E}[g_{\ast}(X)f^{\prime
}(X)-Xf(X)]
\]
which is equivalent to
\[
\mathbf{P}[X>z]-\Phi_{\ast}(z)=\mathbf{E}[Xf(X)-g_{\ast}(X)f^{\prime}(X)].
\]

Since $g_{\ast}(X)\geq0$ almost surely and $f^{\prime}(x)\leq0$ if $x>z$,
\begin{align*}
\mathbf{P}[X>z]-\Phi_{\ast}(z)  &  =\mathbf{E}[\mathbf{1}_{X\leq
z}Xf(X)]+\mathbf{E}[\mathbf{1}_{X>z}Xf(X)]-\mathbf{E}[\mathbf{1}_{X\leq
z}g_{\ast}(X)f^{\prime}(X)]-\mathbf{E}[\mathbf{1}_{X>z}g_{\ast}(X)f^{\prime
}(X)]\\
&  \geq\mathbf{E}[\mathbf{1}_{X\leq z}Xf(X)]+\mathbf{E}[\mathbf{1}%
_{X>z}Xf(X)]-\mathbf{E}[\mathbf{1}_{X\leq z}g_{\ast}(X)f^{\prime}(X)]
\end{align*}

Let $m(x)=[f(a)-f(z)]\mathbf{1}_{x\leq a}+[f(x)-f(z)]\mathbf{1}_{a<x\leq z}$
where the first term is $0$ if $a=-\infty$. Note that $m$ is continuous and
piecewise differentiable. The derivative is $m^{\prime}(x)=f^{\prime
}(x)\mathbf{1}_{a<x\leq z}$ except at $x=a$ and $x=z$. We saw in Lemma
\ref{lem:fprimebounded} that $f^{\prime}$ is bounded. Therefore, since
$X\in\mathbf{D}^{1,2}$, we can use Lemma \ref{lem:IntParts} to conclude that
\[
\lbrack f(a)-f(z)]\mathbf{E}[\mathbf{1}_{X\leq a}X]+\mathbf{E}[\mathbf{1}%
_{a<X\leq z}X(f(X)-f(z))]=\mathbf{E}[\mathbf{1}_{a<X\leq z}f^{\prime}(X)G].
\]
from which we derive
\[
\mathbf{E}[\mathbf{1}_{X\leq z}Xf(X)]-f(z)\mathbf{E}[\mathbf{1}_{X\leq
z}X]=\mathbf{E}[\mathbf{1}_{X\leq z}f^{\prime}(X)G].
\]
Therefore,%
\begin{align*}
\mathbf{P}[X>z]-\Phi_{\ast}(z)  &  \geq\left\{  \mathbf{E}[\mathbf{1}_{X\leq
z}f^{\prime}(X)G]+f(z)\mathbf{E}[\mathbf{1}_{X\leq z}X]\right\}
+\mathbf{E}[\mathbf{1}_{X>z}Xf(X)]-\mathbf{E}[\mathbf{1}_{X\leq z}g_{\ast
}(X)f^{\prime}(X)]\\
&  =\mathbf{E}[\mathbf{1}_{X\leq z}f^{\prime}(X)(G-g_{\ast}%
(X))]+f(z)\mathbf{E}[\mathbf{1}_{X\leq z}X]+\mathbf{E}[\mathbf{1}%
_{X>z}Xf(X)]\\
&  \geq f(z)\mathbf{E}[\mathbf{1}_{X\leq z}X]+\mathbf{E}[\mathbf{1}%
_{X>z}Xf(X)]\\
&  =f(z)\mathbf{E}[\mathbf{1}_{X\leq z}X]+\mathbf{E}[\mathbf{1}_{X>z}%
Xf(X)]-f(z)\mathbf{E}[\mathbf{1}_{X>z}X]+f(z)\mathbf{E}[\mathbf{1}_{X>z}X]\\
&  =f(z)\mathbf{E}[X]+\mathbf{E}[\mathbf{1}_{X>z}X(f(X)-f(z))]\\
&  =\mathbf{E}[\mathbf{1}_{X>z}X(f(X)-f(z))]
\end{align*}

Write $f(X)-f(z)=f^{\prime}(\xi)(X-z)$ for some random $\xi>z$ ($X>\xi$ also).
Note that $f^{\prime}(\xi)<0$ since $\xi>z$. We have $\mathbf{P}%
[X>z]-\Phi_{\ast}(z)\geq\mathbf{E}[\mathbf{1}_{X>z}f^{\prime}(\xi)X(X-z)]$.
From Lemma \ref{lem:fprimebounded},
\[
f^{\prime}(\xi)\geq-\frac{1}{Q(z)}%
\]
since from Lemma \ref{lemQ}, $Q$ is non-decreasing on $\left(  0,b\right)  $.

If we define $S(z):=\mathbf{P}[X>z]$, it is elementary to show (see
\cite{Viens}) that
\[
\mathbf{E}[\mathbf{1}_{X>z}X(X-z)]\leq\int_{z}^{b}(2x-z)S(x)\,dx.
\]
From $\mathbf{P}[X>z]-\Phi_{\ast}(z)\geq\mathbf{E}[\mathbf{1}_{X>z}f^{\prime
}(\xi)X(X-z)]$,
\[
\mathbf{P}[X>z]\geq\Phi_{\ast}(z)-\frac{1}{Q(z)}\int_{z}^{b}(2x-z)S(x)\,dx
\]
which is the statement of the theorem.

Lastly the reader will check that the assumption that the supports of $Z$ and
$X$ have the same left-endpoint is not a restriction: stated briefly, this
assumption is implied by the assumption $G$ $\geq g_{\ast}\left(  X\right)  $
a.s., because $G=g_{X}\left(  X\right)  $ and $g_{\ast}$ (resp. $g_{X}$) has
the same support as $Z$ (resp. $X$).
\end{proof}

\bigskip

To obtain a similar upper bound result, we will consider only asymptotic
statements for $z$ near $b$, and will need an assumption about the relative
growth rate of $g_{\ast}$ and $Q$ near $b$. We will see in the next section
that this assumption is satisfied for all members of the Pearson class with
four moments, although that section also contains a modification of the proof
below which is more efficient when applied to the Pearson class.

\begin{theorem}
\label{thm:TailUBound} Assume all the conditions of Theorem
\ref{thm:TailLBound}\ hold, except for the support of $X$, which we now assume
is contained in $\left(  a,b\right)  $. Assume moreover that there exists
$c<1$ such that $\limsup_{z\rightarrow b}g_{\ast}\left(  z\right)  /Q\left(
z\right)  <c$. If $G\leq g_{\ast}\left(  X\right)  $ a.s., then there exists
$z_{0}$ such that $b>z>z_{0}$ implies%
\[
\mathbf{P}[X>z]\leq\frac{1}{1-c}\Phi_{\ast}(z).
\]

\end{theorem}

\begin{proof}
From Stein's equation (\ref{stein}), and its application (\ref{stein_exp}),%
\[
\mathbf{P}[X>z]-\Phi_{\ast}(z)=\mathbf{E}[Xf(X)-g_{\ast}(X)f^{\prime}(X)].
\]

Since $X\in\mathbf{D}^{1,2}$, in Lemma \ref{lem:IntParts}, we can let $m=f$
since $f$ is continuous, differentiable everywhere except at $x=a$ and $x=b$,
and from Lemma \ref{lem:fprimebounded} has a bounded derivative. Therefore,%
\begin{align*}
\mathbf{P}[X>z]-\Phi_{\ast}(z)  &  =\mathbf{E}[Gf^{\prime}(X)]-\mathbf{E}%
[g_{\ast}(X)f^{\prime}(X)]\\
&  =\mathbf{E}[f^{\prime}(X)\left(  G-g_{\ast}(X)\right)  ]\\
&  =\mathbf{E}[\mathbf{1}_{X\leq z}f^{\prime}(X)\left(  G-g_{\ast}(X)\right)
]+\mathbf{E}[\mathbf{1}_{X>z}f^{\prime}(X)\left(  G-g_{\ast}(X)\right)  ]\\
&  \leq\mathbf{E}[\mathbf{1}_{X>z}f^{\prime}(X)\left(  G-g_{\ast}(X)\right)
]\\
&  =\mathbf{E}[\mathbf{1}_{X>z}f^{\prime}(X)\mathbf{E}\left[  G|X\right]
]-\mathbf{E}[\mathbf{1}_{X>z}f^{\prime}(X)g_{\ast}(X)]
\end{align*}
where the last inequality follows from the assumption $G-g_{\ast}(X)\leq0$
a.s. and if $X\leq z$, then $f^{\prime}(X)\geq0$. By Proposition 3.9 in
\cite{NP}, $\mathbf{E}\left[  G|X\right]  \geq0$ a.s. Since $f^{\prime}%
(X)\leq0$ if $X>z$, then by the last statement in Lemma
\ref{lem:fprimebounded}, and the assumption on the asymptotic behavior of
$g_{\ast}/Q$, for $z$ large enough,%
\begin{align}
\mathbf{P}[X>z]-\Phi_{\ast}(z)  &  \leq-\mathbf{E}[\mathbf{1}_{X>z}f^{\prime
}(X)g_{\ast}(X)]\label{starthere}\\
&  \leq\mathbf{E}\left[  \mathbf{1}_{X>z}\frac{g_{\ast}(X)}{Q\left(  X\right)
}\right] \nonumber\\
&  \leq c\mathbf{P}[X>z].\nonumber
\end{align}
The theorem immediately follows.
\end{proof}


\section{Pearson Distributions}

By definition, the law of a random variable $Z$ is a member of the Pearson
family of distributions if $Z$'s density $\rho_{\ast}$ is characterized by the
differential equation $\rho_{\ast}^{\prime}(z)/\rho_{\ast}(z)=(a_{1}%
z+a_{0})/(\alpha z^{2}+\beta z+\gamma)$ for $z$ in its support $(a,b)$, where
$-\infty\leq a<b\leq\infty$. If furthermore $\mathbf{E}[Z]=0$, Stein (Theorem
1, p. 65 in \cite{Stein}) proved that $g_{\ast}$ has a simple form: in fact,
it is quadratic in its support. Specifically, $g_{\ast}(z)=\alpha z^{2}+\beta
z+\gamma$ for all $z\in\left(  a,b\right)  $ if and only if
\begin{equation}
\frac{\rho_{\ast}^{\prime}(z)}{\rho_{\ast}(z)}=-\frac{(2\alpha+1)z+\beta
}{\alpha z^{2}+\beta z+\gamma}. \label{pearson2}%
\end{equation}

The Appendix contains a description of various cases of Pearson distributions,
which are characterized by their first four moments, if they exist. In this
section, we will operate under the following.

\begin{description}
\item[Assumption P1] Our Pearson random variable satisfies $\mathbf{E}\left[
Z^{2}\right]  <\infty$ and $z^{3}\rho_{\ast}(z)\rightarrow0$ as $z\rightarrow
a$ and $z\rightarrow b$.
\end{description}

\begin{remark}
\label{rem:ass1}This assumption holds as soon as $\mathbf{E}\left[
Z^{3}\right]  <\infty$, which, by Lemma \ref{lem:PearsonMoment} in the
Appendix, holds if and only if $\alpha<1/2$. The existence of a second moment,
by the same lemma, holds if and only if $\alpha<1$.
\end{remark}

\subsection{General comparisons to Pearson tails}

In preparation to stating corollaries to our main theorems, applicable to all
Pearson $Z$'s simultaneously, we begin by investigating the specific
properties of $g_{\ast}$ and $Q$ in the Pearson case. Because $g_{\ast
}(z)=\left(  \alpha z^{2}+\beta z+\gamma\right)  \mathbf{1}_{(a,b)}(z)$, we
have the following observations:

\begin{itemize}
\item Since $g_{\ast}^{\prime\prime}(z)=2\alpha$ on $(a,b)$, and $\alpha<1$
according to Remark \ref{rem:ass1}, it follows that $g_{\ast}^{\prime\prime
}(z)<2$.

\item If $z\in(a,b)$, then
\[
Q(z)=z^{2}-zg_{\ast}^{\prime}(z)+g_{\ast}(z)=z^{2}-z\left(  2\alpha
z+\beta\right)  +\left(  \alpha z^{2}+\beta z+\gamma\right)  =\left(
1-\alpha\right)  z^{2}+\gamma
\]
and so $Q(z)\geq Q(0)=\gamma=g_{\ast}\left(  0\right)  >0$, where the last
inequality is because $g_{\ast}$ is strictly positive on the interior of its
support, which always contains $0$. This is a quantitative confirmation of an
observation made earlier about the positivity of $Q$ in the general case.

\item As $z\rightarrow a$ and $z\rightarrow b$,
\[
\frac{z-g_{\ast}^{\prime}(z)}{Q(z)}=\frac{\left(  1-2\alpha\right)  z-\beta
}{\left(  1-\alpha\right)  z^{2}+\gamma}%
\]
approaches a finite number in case $a$ and $b$ are finite. As $|z|\rightarrow
\infty$, the above ratio approaches $0$.

\item We have $\mathbf{E}\left[  Z^{2}\right]  =\frac{\gamma}{1-\alpha}$.
Again, this is consistent with $\gamma>0$ and $\alpha<1$.\label{bullets}
\end{itemize}

\begin{remark}
The above observations collectively mean that all the assumptions of Theorem
\ref{thm:TailLBound} are satisfied for our Pearson random variable $Z$, so we
can state the following.
\end{remark}

\begin{proposition}
\label{thm:TailLBoundP} Let $Z$ be a centered Pearson random variable
satisfying Assumption P1. Let $g_{\ast}$ be defined as in (\ref{geestar}). Let
$X\in\mathbf{D}^{1,2}$ be a centered random variable with a density, and whose
support $\left(  a,b_{X}\right)  $ contains $(a,b)$. Suppose that $G\geq
g_{\ast}\left(  X\right)  $ a.s. Then for every $z\in(0,b)$,
\begin{equation}
\mathbf{P}[X>z]\geq\Phi_{\ast}(z)-\frac{1}{\left(  1-\alpha\right)
z^{2}+\gamma}\int_{z}^{b}(2x-z)\mathbf{P}[X>x]\,dx. \label{result1P}%
\end{equation}

\end{proposition}

We have a quantitatively precise statement on the relation between
$\operatorname*{Var}[X]$ and the Pearson parameters.

\begin{proposition}
~

\begin{enumerate}
\item Assume that the conditions of Proposition \ref{thm:TailLBoundP} hold,
particularly that $G\geq g_{\ast}(X)$; assume the support $\left(  a,b\right)
$ of $g_{\ast}$ coincides with the support of $X$. Then%
\[
\operatorname*{Var}[X]\geq\frac{\gamma}{1-\alpha}=\operatorname*{Var}[Z].
\]

\item If we assume instead that $G\leq g_{\ast}(X)$ a.s., then the inequality
above is reversed.
\end{enumerate}
\end{proposition}

\begin{proof}
Since $X$ has a density, we can apply Lemma \ref{lem:IntParts} and let
$m(x)=x$.%
\begin{align*}
\operatorname*{Var}[X]  &  =\mathbf{E}[Xm(X)]=\mathbf{E}[G]\geq\mathbf{E}%
[g_{\ast}(X)]\\
&  \geq\mathbf{E}[g_{\ast}(X)]=\mathbf{E}[\mathbf{1}_{a<X<b}\left(  \alpha
X^{2}+\beta X+\gamma\right)  ]=\alpha\mathbf{E}[X^{2}]+\beta\mathbf{E}%
[X]+\gamma\\
\left(  1-\alpha\right)  \operatorname*{Var}[X]  &  \geq\beta\cdot0+\gamma\\
\operatorname*{Var}[X]  &  \geq\frac{\gamma}{1-\alpha}=\operatorname*{Var}[Z].
\end{align*}
This proves point 1. Point 2 is done identically.
\end{proof}

\bigskip

In order to formulate results that are specifically tailored to tail
estimates, we now make the following.

\begin{description}
\item[Assumption P2] The right-hand endpoint of our Pearson distribution's
support is $b=+\infty$
\end{description}

\begin{remark}
Assumption P2 leaves out Case 3 in the Appendix in our list of Pearson random
variables, i.e. the case of Beta distributions. Therefore, inspecting the
parameter values in the other 4 Pearson cases, we see that Assumption P2
implies $\alpha\geq0$, and also implies that if $\alpha=0$, then $\beta\geq0$.
\end{remark}

\begin{remark}
In most of the results to follow, we will assume moreover that $\alpha
<\frac{1}{2}$. By Lemma \ref{lem:PearsonMoment}, this is equivalent to
requiring $\mathbf{E}\left[  \left\vert Z\right\vert ^{3}\right]  <\infty$,
and more generally from the lemma, our Pearson distribution has moment of
order $m$ if and only if $\alpha<1/(m-1)$. As mentioned, $\alpha<\frac{1}{2}$
thus implies Assumption P1. Consequently Theorem \ref{thm:TailUBound} implies
the following.
\end{remark}

\begin{corollary}
\label{thm:TailUBoundP2} Let $Z$ be a centered Pearson random variable
satisfying Assumption P2 (support of $Z$ is $(a,+\infty)$).\ Assume
$\alpha<1/2$. Let $g_{\ast}$ be defined as in (\ref{geestar}). Let
$X\in\mathbf{D}^{1,2}$ be a centered random variable with a density and
support contained in $(a,+\infty)$. If $G$ $\leq g_{\ast}\left(  X\right)  $
a.s., for any $K>\frac{1-\alpha}{1-2\alpha}$, there exists $z_{0}$ such that
if $z>z_{0}$, then
\[
\mathbf{P}[X>z]\leq K~\Phi_{\ast}(z).
\]

\end{corollary}

\begin{proof}
Since
\[
\frac{g_{\ast}\left(  z\right)  }{Q\left(  z\right)  }=\frac{\alpha
z^{2}+\beta z+\gamma}{\left(  1-\alpha\right)  z^{2}+\gamma}%
\]
then $\limsup_{z\rightarrow\infty}g_{\ast}\left(  z\right)  /Q\left(
z\right)  =\alpha/\left(  1-\alpha\right)  <1$ if and only if $\alpha<\frac
{1}{2}$. Therefore, Theorem \ref{thm:TailUBound} applies in this case, and
with the $c$ defined in that theorem, we may take here any $c>\alpha/\left(
1-\alpha\right)  $, so that we may take any $K=1/(1-c)$ as announced.
\end{proof}

\bigskip

The drawback of our general lower bound theorems so far is that their
statements are somewhat implicit. Our next effort is to fix this problem in
the specific case of a Pearson $Z$: we strengthen Proposition
\ref{thm:TailLBoundP} so that the tail $\mathbf{P}\left[  X>z\right]  $ only
appears in the left-hand side of the lower bound inequality, making the bound
explicit. The cost for this is an additional regularity and integrability
assumption, whose scope we also discuss.

\begin{corollary}
\label{cor:TailLBoundP} Assume that the conditions of Proposition
\ref{thm:TailLBoundP} hold; in particular, assume $X\in\mathbf{D}^{1,2}$ and
$G\geq\alpha X^{2}+\beta X+\gamma$ a.s. In addition, assume there exists a
constant $c>2$ such that $\mathbf{P}\left[  X>z\right]  \leq z\rho\left(
z\right)  /c$ holds for large $z$ (where $\rho$ is the density of $X$). Then
for large $z$,
\[
\mathbf{P}[X>z]\geq\frac{\left(  c-2\right)  Q(z)}{\left(  c-2\right)
Q(z)+2z^{2}}\Phi_{\ast}(z)\approx\frac{\left(  c-2\right)  -\alpha\left(
c-2\right)  }{c-\alpha\left(  c-2\right)  }\Phi_{\ast}(z).
\]

The existence of such a $c>2$ above is guaranteed if we assume $g(z)\leq
z^{2}/c$ for large $z$, where $g\left(  x\right)  :=\mathbf{E}\left[
G|X=x\right]  $ (or equivalently, $g$ defined in (\ref{geeintro})). Moreover,
this holds automatically if $G\leq\bar{g}_{\ast}\left(  X\right)  $ a.s. for
some quadratic function $\bar{g}_{\ast}\left(  x\right)  =\bar{\alpha}%
x^{2}+\bar{\beta}x+\bar{\gamma}$ with $\bar{\alpha}<1/2$.
\end{corollary}

\begin{proof}
Since $z>0$, we can replace $2x-z$ by $2x$ in the integral of (\ref{result1P}%
).%
\begin{align*}
F(z):=  &  \int_{z}^{\infty}xS(x)\,dx\leq\frac{1}{c}\int_{z}^{\infty}%
x^{2}|S^{\prime}(x)|\,dx\\
&  =\frac{1}{c}\left(  z^{2}S(z)-\lim_{x\rightarrow\infty}x^{2}S(x)+2\int
_{z}^{\infty}xS(x)\,dx\right) \\
&  \leq\frac{1}{c}\left(  z^{2}S(z)+2F(z)\right) \\
F(z)  &  \leq\frac{1}{c-2}z^{2}S(z)
\end{align*}
Therefore%
\begin{align*}
S(z)=\mathbf{P}[X>z]  &  \geq\Phi_{\ast}(z)-\frac{2}{Q(z)}F(z)\geq\Phi_{\ast
}(z)-\frac{2z^{2}}{\left(  c-2\right)  Q(z)}S(z)\\
S(z)\left[  1+\frac{2z^{2}}{\left(  c-2\right)  Q(z)}\right]   &  \geq
\Phi_{\ast}(z)\\
S(z)  &  \geq\frac{\left(  c-2\right)  Q(z)}{\left(  c-2\right)  Q(z)+2z^{2}%
}\Phi_{\ast}(z)\approx\frac{\left(  c-2\right)  (1-\alpha)}{\left(
c-2\right)  (1-\alpha)+2}\Phi_{\ast}(z).
\end{align*}
This proves the inequality of the Corollary.

To prove the second statement, recall that \cite[Theorem 3.1]{NV} showed that%
\[
g(X)=\frac{\displaystyle\int_{X}^{\infty}x\rho(x)\,dx}{\rho(X)}%
\]
$\mathbf{P}$-a.s. It is also noted therein that the support of $\rho$ is an
interval since $X\in\mathbf{D}^{1,2}$. Therefore,
\[
\frac{z}{c}\rho(z)\geq\frac{1}{z}g(z)\rho(z)=\int_{z}^{\infty}\frac{x}{z}%
\rho(x)\,dx\geq\int_{z}^{\infty}\rho(x)\,dx
\]
a.s. This finishes the proof of the corollary.
\end{proof}

\subsection{Comparisons in specific scales}

In this section and the next, we will always assume $X\in\mathbf{D}^{1,2}$ is
a centered random variable with a density and with support $\left(
a,\infty\right)  $, and we will continue to denote by $g$ the function defined
by $g\left(  x\right)  :=\mathbf{E}\left[  G|X=x\right]  $, or equivalently,
defined in (\ref{geeintro}).

We can exploit the specific asymptotic behavior of the tail of the various
Pearson distributions, via Lemma \ref{lem:phitail} in the Appendix, to draw
sharp conclusions about $X$'s tail. For instance, if $g$ is comparable to a
Pearson distribution's $g_{\ast}$ with $\alpha\neq0$, we get a power decay for
the tail (Corollary \ref{cor:TailX2} below), while if $\alpha$ is zero and
$\beta$ is not, we get comparisons to exponential-type or gamma-type tails
(Corollary \ref{cor:TailX1} below). In both cases, when upper and lower bounds
on $G$ occur with the same $\alpha$ on both sides, we get sharp asymptotics
for $X$'s tail, up to multiplicative constants.

\begin{corollary}
\label{cor:TailX2}Let $g_{\ast}\left(  x\right)  :=\alpha x^{2}+\beta
x+\gamma$ and $\bar{g}_{\ast}\left(  x\right)  :=\bar{\alpha}x^{2}+\bar{\beta
}x+\bar{\gamma}$ be two functions corresponding to Pearson distributions (e.g.
via (\ref{geestar})) where $0<\alpha\leq\bar{\alpha}<1/2$.

\begin{enumerate}
\item If $g\left(  x\right)  \leq\bar{g}_{\ast}\left(  x\right)  $ for all
$x\geq a$, then there is a constant $c_{u}\left(  \bar{\alpha},\bar{\beta
},\bar{\gamma}\right)  >0$ such that for large $z$,
\[
\mathbf{P}[X>z]\leq\frac{c_{u}}{z^{1+1/\bar{\alpha}}}.
\]

\item If $g_{\ast}\left(  x\right)  \leq g\left(  x\right)  \leq\bar{g}_{\ast
}\left(  x\right)  $ for all $x\geq a$, then there are constants $c_{u}\left(
\bar{\alpha},\bar{\beta},\bar{\gamma}\right)  >0$ and $c_{l}\left(
\bar{\alpha},\alpha,\beta,\gamma\right)  >0$ such that for large $z$,
\[
\frac{c_{l}}{z^{1+1/\alpha}}\leq\mathbf{P}[X>z]\leq\frac{c_{u}}{z^{1+1/\bar
{\alpha}}}.
\]

\end{enumerate}
\end{corollary}

\begin{proof}
Let $\Phi_{\ast\alpha,\beta,\gamma}$ and $\Phi_{\ast\bar{\alpha},\bar{\beta
},\bar{\gamma}}$ be the probability tails of the Pearson distributions
corresponding to $g_{\ast}$ and $\bar{g}_{\ast}$ respectively. We can prove
Point 1 by using Corollary \ref{thm:TailUBoundP2} and Lemma \ref{lem:phitail}.
There is a constant $k_{u}\left(  \bar{\alpha},\bar{\beta},\bar{\gamma
}\right)  >0$ such that, for any $K>\frac{1-\bar{\alpha}}{1-2\bar{\alpha}}$,
for large $z$,
\[
\mathbf{P}[X>z]\leq K\Phi_{\ast\bar{\alpha},\bar{\beta},\bar{\gamma}}\left(
z\right)  \leq K\cdot\frac{k_{u}}{z^{1+1/\bar{\alpha}}}\text{.}%
\]
The upper bound in Point 2 follows directly from Point 1 because of the
condition $g\left(  x\right)  \leq\bar{g}_{\ast}\left(  x\right)  $. This same
condition also allows us to give a lower bound for $\mathbf{P}[X>z]$. Fix any
$c\in\left(  2,1/\bar{\alpha}\right)  $. By Corollary \ref{cor:TailLBoundP}
and Lemma \ref{lem:phitail}, there is a constant $k_{l}\left(  \alpha
,\beta,\gamma\right)  >0$ such that for large $z$,
\[
\mathbf{P}[X>z]\geq\frac{\left(  c-2\right)  -\alpha\left(  c-2\right)
}{c-\alpha\left(  c-2\right)  }\Phi_{\ast\alpha,\beta,\gamma}\left(  z\right)
\geq\frac{\left(  c-2\right)  -\alpha\left(  c-2\right)  }{c-\alpha\left(
c-2\right)  }\cdot\frac{k_{l}}{z^{1+1/\alpha}}\text{.}%
\]

\end{proof}

\begin{corollary}
\label{cor:TailX1}Let $g_{\ast}\left(  x\right)  :=\left(  \beta
x+\gamma\right)  _{+}$ and $\bar{g}_{\ast}\left(  x\right)  =\left(
\bar{\beta}x+\bar{\gamma}\right)  _{+}$ be two functions corresponding to
Pearson distributions (e.g. via (\ref{geestar})) where $\beta,\bar{\beta
},\gamma,\Bar{\gamma}>0$ and $a=-\gamma/\beta$.

\begin{enumerate}
\item If $g\left(  x\right)  \leq\bar{g}_{\ast}\left(  x\right)  $ for all
$x$, then there is a constant $c_{u}\left(  \bar{\beta},\bar{\gamma}\right)
>0$ such that for large $z$,
\[
\mathbf{P}[X>z]\leq c_{u}z^{-1-\bar{\gamma}/\bar{\beta}^{2}}e^{-z/\bar{\beta}%
}.
\]

\item If $g_{\ast}\left(  x\right)  \leq g\left(  x\right)  \leq\bar{g}_{\ast
}\left(  x\right)  $ for all $x$, then there are constants $c_{u}\left(
\bar{\beta},\bar{\gamma}\right)  >c_{l}\left(  \beta,\gamma\right)  >0$ such
that for large $z$,
\[
c_{l}~z^{-1-\gamma/\beta^{2}}e^{-z/\beta}\leq\mathbf{P}[X>z]\leq
c_{u}z^{-1-\bar{\gamma}/\bar{\beta}^{2}}e^{-z/\bar{\beta}}.
\]

\end{enumerate}
\end{corollary}

\begin{proof}
Let $\Phi_{\ast\beta,\gamma}$ and $\Phi_{\ast\bar{\beta},\bar{\gamma}}$ be as
in the proof of the previous corollary, noting here that $\alpha=\bar{\alpha
}=0$. The proof of Point 1 is similar to the proof of Point 1 in Corollary
\ref{cor:TailX2}. The upper bound in Point 2 follows from Point 1 above and
Point 1 of Corollary \ref{cor:TailX2}. For the lower bound of Point 2, if we
fix any $c>2$, then by Corollary \ref{cor:TailLBoundP} and Lemma
\ref{lem:phitail}, there is a constant $k_{l}\left(  \beta,\gamma\right)  >0$
such that for large $z$,
\[
\mathbf{P}[X>z]\geq\frac{c-2}{c}\Phi_{\ast\beta,\gamma}\geq\frac{c-2}{c}%
~k_{l}~z^{-1-\gamma/\beta^{2}}e^{-z/\beta}\text{.}%
\]

\end{proof}

\bigskip

\begin{remark}
The above corollary improves on a recently published estimate: in
\cite[Theorem 4.1]{NV}, it was proved that if the law of $X\in\mathbf{D}%
^{1,2}$ has a density and if $g(X)\leq\beta X+\gamma$ a.s. (with $\beta\geq0$
and $\gamma>0$), then for all $z>0$, $\mathbf{P}[X>z]\leq\exp\left(
-\frac{z^{2}}{2\beta z+2\gamma}\right)  $. Using $g_{\ast}\left(  z\right)
=\left(  \beta z+\gamma\right)  _{+}$, Point 1 in Corollary \ref{cor:TailX1}
gives us an asymptotically better upper bound, with exponential rate
$e^{-z/\beta}$ instead of $e^{-z/2\beta}$. Our rate is sharp, since our upper
bound has the same exponential asymptotics as the corresponding Pearson tail,
which is a Gamma tail.
\end{remark}

\subsection{Asymptotic results}

Point 2 of Corollary \ref{cor:TailX1} shows the precise behavior, up to a
possibly different leading power term which is negligible compared to the
exponential, of any random variable in $\mathbf{D}^{1,2}$ whose function $g$
is equal to a Pearson function up to some uncertainty on the $\gamma$ value.
More generally, one can ask about tail asymptotics for $X$ when $g$ is
asymptotically linear, or even asymptotically quadratic. Asymptotic
assumptions on $g$ are not as strong as assuming bounds on $g$ which are
uniform in the support of $X$, and one cannot expect them to imply statements
that are as strong as in the previous subsection. We now see that in order to
prove tail asymptotics under asymptotic assumptions, it seems preferable to
revert to the techniques developed in \cite{Viens}. We first propose upper
bound results for tail asymptotics, which follow from Point 1 of Corollary
\ref{cor:TailX2} and Point 1 of Corollary \ref{cor:TailX1}. Then for full
asymptotics, Point 2 of each of these corollaries do not seem to be
sufficient, while \cite[Corollary 4.5]{Viens} can be applied immediately.
Recall that in what follows $X\in\mathbf{D}^{1,2}$ is centered, has a density,
and support $\left(  a,\infty\right)  $, and $g$ is defined by $g\left(
x\right)  :=\mathbf{E}\left[  G|X=x\right]  $, or equivalently, by
(\ref{geeintro}).

\begin{proposition}
~

\begin{enumerate}
\item Suppose $\limsup_{z\rightarrow+\infty}g\left(  z\right)  /z^{2}%
=\alpha\in\left(  0,1/2\right)  $. Then $\limsup_{z\rightarrow+\infty}%
\frac{\ln\mathbf{P}[X>z]}{\ln z}\leq-\left(  1+\frac{1}{\alpha}\right)  .$

\item Suppose $\limsup_{z\rightarrow+\infty}g\left(  z\right)  /z=\beta>0$.
Then $\limsup_{z\rightarrow+\infty}\frac{\ln\mathbf{P}[X>z]}{z}\leq-\frac
{1}{\beta}.$
\end{enumerate}
\end{proposition}

\begin{proof}
Fix $\varepsilon\in\left(  0,1/2-\alpha\right)  $. Then $g\left(  x\right)
<\left(  \alpha+\varepsilon\right)  x^{2}$ if $x$ is large enough. Therefore,
there exists a constant $\gamma_{\varepsilon}>0$ such that $g\left(  x\right)
<\left(  \alpha+\varepsilon\right)  x^{2}+\gamma_{\varepsilon}$ for all
$x$.\ Let $Z_{\varepsilon}$ be the Pearson random variable for which $g_{\ast
}\left(  z\right)  =\left(  \alpha+\varepsilon\right)  z^{2}+\gamma
_{\varepsilon}$. This falls under Case 5 in Appendix \ref{PearsonList}, so its
support is $\left(  -\infty,\infty\right)  $, which then contains the support
of $X$. From Point 1 of Corollary \ref{cor:TailX2}, there is a constant
$c_{\varepsilon}$ depending on $\varepsilon$ such that for $z$ large enough,
\[
\mathbf{P}[X>z]\leq c_{\varepsilon}z^{-1-\frac{1}{\alpha+\varepsilon}}\text{.}%
\]
We then have
\begin{align*}
\ln\mathbf{P}[X>z]  &  \leq\ln c_{\varepsilon}-\left(  1+\frac{1}%
{\alpha+\varepsilon}\right)  \ln z,\\
\frac{\ln\mathbf{P}[X>z]}{\ln z}  &  \leq\frac{\ln c_{\varepsilon}}{\ln
z}-\left(  1+\frac{1}{\alpha+\varepsilon}\right)  ,\\
\limsup_{z\rightarrow\infty}\frac{\ln\mathbf{P}[X>z]}{\ln z}  &  \leq-\left(
1+\frac{1}{\alpha+\varepsilon}\right)  \text{.}%
\end{align*}
Since $\varepsilon$ can be arbitrarily close to $0$, Point 1 of the corollary
is proved. The proof of Point 2 is entirely similar, following from Corollary
\ref{cor:TailX1}, which refers to Case 2 of the Pearson distributions given in
Appendix \ref{PearsonList}. This corollary could also be established by using
results from \cite{Viens}.
\end{proof}

\bigskip

Our final result gives full tail asymptotics. Note that it is not restricted
to linear and quadratic behaviors.

\begin{theorem}
~

\begin{enumerate}
\item Suppose $\lim_{z\rightarrow+\infty}g\left(  z\right)  /z^{2}=\alpha
\in\left(  0,1\right)  $. Then $\lim_{z\rightarrow+\infty}\frac{\ln
\mathbf{P}[X>z]}{\ln z}=-\left(  1+\frac{1}{\alpha}\right)  .$

\item Suppose $\lim_{z\rightarrow+\infty}g\left(  z\right)  /z^{p}=\beta>0$
for some $p\in\lbrack0,1)$. Then $\lim_{z\rightarrow+\infty}\frac
{\ln\mathbf{P}[X>z]}{z^{2-p}}=-\frac{1}{\beta(2-p)}.$
\end{enumerate}
\end{theorem}

\begin{proof}
Since for any $\varepsilon\in(0,\min(\alpha,1-\alpha))$, there exists $z_{0}$
such that $z>z_{0}$ implies $\left(  \alpha-\varepsilon\right)  z^{2}\leq
g\left(  z\right)  \leq\left(  \alpha+\varepsilon\right)  z^{2}$, the
assumptions of Points 2 and 4 (a) in \cite[Corollary 4.5]{Viens} are
satisfied, and Point 1 of the Theorem follows easily. Point 2 of the Theorem
follows identically, by invoking Points 3 and 4 (b) in \cite[Corollary
4.5]{Viens}. All details are left to the reader.
\end{proof}

\section{Appendix}

\subsection{Proofs of lemmas}

\begin{proof}
[Proof of lemma \ref{lemgstar}]\emph{Proof of point 1}. If $0\leq x<b$, then
clearly $g_{\ast}(x)>0$. If $a<x<0$, we claim that $g_{\ast}(x)>0$ still.
Suppose we have the opposite: $g_{\ast}(x)\leq0$. Then $\int_{x}^{b}%
y\rho_{\ast}(y)\,dy=g_{\ast}(x)\rho_{\ast}(x)\leq0.$ Since $\int_{a}^{x}%
y\rho_{\ast}(y)\,dy<0$, then $\int_{a}^{b}y\rho_{\ast}(y)\,dy<0$,
contradicting $\mathbf{E}[Z]=0.$ Thus, $g_{\ast}(x)\geq0$ for all $x$, and
$g_{\ast}(x)>0$ if and only if $a<x<b$.

\emph{Proof of point 2}. Trivial.

\emph{Proof of point 3}. This is immediate since $\lim\limits_{x\downarrow
a}g_{\ast}(x)\rho_{\ast}(x)=\lim\limits_{x\downarrow a}\int_{x}^{b}y\rho
_{\ast} (y)\,dy=\mathbf{E}[Z]$ and similarly for $\lim\limits_{x\uparrow
b}g_{\ast}(x)\rho_{\ast}(x)=-\mathbf{E}[Z]$.
\end{proof}

\begin{proof}
[Proof of Lemma \ref{lemschou}]It is easy to verify that (\ref{stein_sol1})
and (\ref{stein_sol2})\ are solutions to Stein's equation (\ref{stein}). To
show that they are the same, let $\varphi(z):=g_{\ast}(z)\rho_{\ast}%
(z)=\int_{z}^{b}w\rho_{\ast}(w)\,dw$ for $z\in(a,b)$. Then
\[
\frac{\varphi^{\prime}(z)}{\varphi(z)}=-\frac{z\rho_{\ast}(z)}{g_{\ast}%
(z)\rho_{\ast} (z)}=-\frac{z}{g_{\ast}(z)}.
\]
Integrating over $(y,x)\subseteq(a,b)$ leads to
\begin{equation}
\int_{y}^{x}\frac{z}{g_{\ast}(z)}\,dz=\log\frac{\varphi(y)}{\varphi(x)}%
=\log\frac{g_{\ast}(y)\rho_{\ast}(y)}{g_{\ast}(x)\rho_{\ast}(x)}
\label{usenext}%
\end{equation}
and so
\[
\frac{e^{\int_{y}^{x}\frac{z\,dz}{g_{\ast}(z)}}}{g_{\ast}(y)}=\frac{1}%
{g_{\ast}(y)}\cdot\frac{g_{\ast}(y)\rho_{\ast}(y)}{g_{\ast}(x)\rho_{\ast}%
(x)}=\frac{\rho_{\ast} (y)}{g_{\ast}(x)\rho_{\ast}(x)}.
\]
The derivative formula (\ref{deriv1}) comes via an immediate calculation%
\begin{align*}
f^{\prime}(x)  &  =-\frac{[g_{\ast}(x)\rho_{\ast}(x)]^{\prime}}{[g_{\ast
}(x)\rho_{\ast}(x)]^{2}}\int_{a}^{x}\left(  h(y)-\mathbf{E}[h(Z)]\right)
\rho_{\ast}(y)\,dy+\frac{\left(  h(x)-\mathbf{E}[h(Z)]\right)  \rho_{\ast}%
(x)}{g_{\ast}(x)\rho_{\ast}(x)}\\
&  =-\frac{-x\rho_{\ast}(x)}{[g_{\ast}(x)\rho_{\ast}(x)]^{2}}\int_{a}%
^{x}\left(  h(y)-\mathbf{E}[h(Z)]\right)  \rho_{\ast}(y)\,dy+\frac
{h(x)-\mathbf{E}[h(Z)]}{g_{\ast}(x)}\\
&  =\frac{x}{[g_{\ast}(x)]^{2}\rho_{\ast}(x)}\int_{a}^{x}\left(
h(y)-\mathbf{E}[h(Z)]\right)  \rho_{\ast}(y)\,dy+\frac{h(x)-\mathbf{E}%
[h(Z)]}{g_{\ast}(x)}.
\end{align*}

\end{proof}

\begin{proof}
[Proof of Lemma \ref{lemchar}]From (\ref{usenext}) in the previous proof, we
have%
\[
\int_{0}^{b}\frac{z}{g_{\ast}(z)}\,dz=\lim_{x\nearrow b}\int_{0}^{x}\frac
{z}{g_{\ast}(z)}\,dz=\lim_{x\nearrow b}\log\frac{g_{\ast}(0)\rho_{\ast}%
(0)}{g_{\ast}(x)\rho_{\ast}(x)}=g_{\ast}(0)\rho_{\ast}(0)-\lim_{x\nearrow
b}\log\left[  g_{\ast}(x)\rho_{\ast}(x)\right]  =\infty
\]
and%
\[
\int_{a}^{0}\frac{z}{g_{\ast}(z)}\,dz=\lim_{x\searrow a}\int_{x}^{0}\frac
{z}{g_{\ast}(z)}\,dz=\lim_{x\searrow a}\log\frac{g_{\ast}(x)\rho_{\ast}%
(x)}{g_{\ast}(0)\rho_{\ast}(0)}=\lim_{x\searrow a}\log\left[  g_{\ast}%
(x)\rho_{\ast}(x)\right]  -g_{\ast}(0)\rho_{\ast}(0)=-\infty\text{.}%
\]

\end{proof}

\begin{proof}
[Proof of Lemma \ref{lem:phibound}]We prove (\ref{tail1a}) first. It is
trivially true if $x\notin\left[  a,b\right]  $, so suppose $x\in\left(
a,b\right)  $. Let%
\[
m(x):=\Phi_{\ast}(x)-\frac{x-g_{\ast}^{\prime}(x)}{Q(x)}\cdot g_{\ast}%
(x)\rho_{\ast}(x).
\]

By a standard calculus proof, we will show that $m^{\prime}(x)\leq0$ so that
$m(x)\geq\lim\limits_{y\rightarrow b}m(y)$. The result follows after observing
that $\lim\limits_{y\rightarrow b}m(y)=0$. This is true since $\lim
\limits_{y\rightarrow b}g_{\ast}(y)\rho_{\ast}(y)=0$ and $\lim
\limits_{y\rightarrow b}\Phi_{\ast}(x)=0$. Now we show that $m^{\prime}%
(x)\leq0$.%

\begin{align*}
m^{\prime}(x)  &  =-\rho_{\ast}(x)-g_{\ast}(x)\rho_{\ast}(x)\left[
\frac{x-g_{\ast}^{\prime}(x)}{x^{2}-xg_{\ast}^{\prime}(x)+g_{\ast}(x)}\right]
^{\prime}-\frac{x-g_{\ast}^{\prime}(x)}{Q(x)}\left[  g_{\ast}(x)\rho_{\ast
}(x)\right]  ^{\prime}\\
m^{\prime}  &  =-\rho_{\ast}-g_{\ast}\rho_{\ast}\left[  \frac{\left(  x\left(
x-g_{\ast}^{\prime}\right)  +g_{\ast}\right)  \left(  1-g_{\ast}^{\prime
\prime}\right)  -\left(  x-g_{\ast}^{\prime}\right)  \left(  2x-xg_{\ast
}^{\prime\prime}\right)  }{Q^{2}}\right]  -\frac{x-g_{\ast}^{\prime}}%
{Q}\left[  -x\rho_{\ast}\right] \\
\frac{Q^{2}}{\rho_{\ast}}m^{\prime}  &  =-Q^{2}-g_{\ast}\left[  \left(
x-g_{\ast}^{\prime}\right)  \left(  x-xg_{\ast}^{\prime\prime}-2x+xg_{\ast
}^{\prime\prime}\right)  +g_{\ast}\left(  1-g_{\ast}^{\prime\prime}\right)
\right]  +Qx\left(  x-g_{\ast}^{\prime}\right) \\
&  =\left[  -x^{2}\left(  x-g_{\ast}^{\prime}\right)  ^{2}-2xg_{\ast}\left(
x-g_{\ast}^{\prime}\right)  -g_{\ast}^{2}\right]  +xg_{\ast}\left(  x-g_{\ast
}^{\prime}\right)  -g_{\ast}^{2}\left(  1-g_{\ast}^{\prime\prime}\right) \\
&  +\left[  x^{2}\left(  x-g_{\ast}^{\prime}\right)  ^{2}+xg_{\ast}\left(
x-g_{\ast}^{\prime}\right)  \right] \\
&  =-g_{\ast}^{2}-g_{\ast}^{2}\left(  1-g_{\ast}^{\prime\prime}\right)
=-g_{\ast}^{2}\left(  2-g_{\ast}^{\prime\prime}\right)  \leq0\text{.}%
\end{align*}

To prove (\ref{tail1b}) (again, it suffices to prove this for $x\in\left(
a,b\right)  $), let
\[
n(x):=1-\Phi_{\ast}(x)-\frac{g_{\ast}^{\prime}(x)-x}{Q(x)}\cdot g_{\ast
}(x)\rho_{\ast}(x)=1-m(x)
\]
so $n^{\prime}(x)=-m^{\prime}(x)\geq0$. $n$ is then nondecreasing so
$n(x)\geq\lim\limits_{x\rightarrow a}n(x)=0$.

Now we prove (\ref{tail2}). If $x>0$,%
\[
\Phi_{\ast}(x)=\int_{x}^{\infty}\rho_{\ast}(y)\,dy\leq\frac{1}{x}\int
_{x}^{\infty}y\rho_{\ast}(y)\,dy=\frac{1}{x}\cdot g_{\ast}(x)\rho_{\ast}(x).
\]
On the other hand, if $x<0$,
\[
1-\Phi_{\ast}(x)=\int_{-\infty}^{x}\rho_{\ast}(y)\,dy\leq\frac{1}{x}%
\int_{-\infty}^{x}y\rho_{\ast}(y)\,dy=-\frac{1}{x}\cdot g_{\ast}(x)\rho_{\ast
}(x).
\]
This proves (\ref{tail3}).
\end{proof}

\begin{proof}
[Proof of last bullet point on page \pageref{bullets}]\label{bulletproof}We
replicate here the method commonly used to find a recursive formula for the
moments. See for example \cite{EJ} and \cite{Ord}. Cross-multiplying the terms
in (\ref{pearson2}), multiplying by $x^{r}$ and integrating over the support
gives us
\begin{align*}
-\int_{a}^{b}\left[  (2\alpha+1)z^{r+1}+\beta z^{r}\right]  \rho_{\ast
}(z)\,dz  &  =\int_{a}^{b}\left(  \alpha z^{r+2}+\beta z^{r+1}+\gamma
z^{r}\right)  \rho_{\ast}^{\prime}(z)\,dz\\
-(2\alpha+1)\mathbf{E}\left[  Z^{r+1}\right]  -\beta\mathbf{E}\left[
Z^{r}\right]   &  =\left.  \left(  \alpha z^{r+2}+\beta z^{r+1}+\gamma
z^{r}\right)  \rho_{\ast}(z)\right\vert _{a}^{b}\\
&  -\int_{a}^{b}\left[  \alpha(r+2)z^{r+1}+\beta(r+1)z^{r}+\gamma
rz^{r-1}\right]  \rho_{\ast}(z)\,dz\\
(2\alpha+1)\mathbf{E}\left[  Z^{r+1}\right]  +\beta\mathbf{E}\left[
Z^{r}\right]   &  =\alpha(r+2)\mathbf{E}\left[  Z^{r+1}\right]  +\beta
(r+1)\mathbf{E}\left[  Z^{r}\right]  +\gamma r\mathbf{E}\left[  Z^{r-1}%
\right]
\end{align*}
where we assumed that $z^{r+2}\rho_{\ast}(z)\rightarrow0$ at the endpoints $a$
and $b$ of the support. For the case $r=1$, this reduces to $z^{3}\rho_{\ast
}(z)\rightarrow0$ at the endpoints $a$ and $b$, which we are assuming.
Therefore,%
\[
(2\alpha+1)\mathbf{E}\left[  Z^{2}\right]  +\beta\mathbf{E}\left[  Z\right]
=3\alpha\mathbf{E}\left[  Z^{2}\right]  +2\beta\mathbf{E}\left[  Z\right]
+\gamma\mathbf{E}\left[  Z^{0}\right]  \text{.}%
\]
Since $\mathbf{E}\left[  Z\right]  =0$ and $\mathbf{E}\left[  Z^{0}\right]
=1$, this gives $\mathbf{E}\left[  Z^{2}\right]  =\frac{\gamma}{1-\alpha}$.
\end{proof}

\subsection{\bigskip Examples of Pearson distributions\label{PearsonList}}

We present cases of Pearson distributions depending on the degree and number
of zeroes of $g_{\ast}(x)$ as a quadratic polynomial in $(a,b)$. The Pearson
family is closed under affine transformations of the random variable, so we
can limit our focus on the five special cases below. The constant $C$ in each
case represents the normalization constant. See Diaconis and Zabell \cite{DZ}
for a discussion of these cases.

\begin{itemize}
\item CASE 1. If $\deg g_{\ast}(z)=0$, $\rho_{\ast}$ can be (after an affine
transformation) written in the form $\rho_{\ast}(z)=Ce^{-z^{2}/2}$ for
$-\infty<z<\infty$. This is the standard normal density, and $C=1/\sqrt{2\pi}%
$. For this case, $g_{\ast}(z)\equiv1$. Consequently, $Q(z)=z^{2}+1$. If
$z>0$, the inequalities (\ref{tail1a}) and (\ref{tail2})\ of Lemma
\ref{lem:phibound} can be written
\[
\frac{z}{\left(  z^{2}+1\right)  \sqrt{2\pi}}e^{-z^{2}/2}\leq\Phi_{\ast
}(z)\leq\frac{1}{z\sqrt{2\pi}}e^{-z^{2}/2},
\]
a standard inequality involving the tail of the standard normal distribution.

\item CASE 2. If $\deg g_{\ast}(z)=1$, $\rho_{\ast}$ can be written in the
form $\rho_{\ast}(z)=Cz^{r-1}e^{-z/s}$ for $0<z<\infty$, with parameters
$r,s>0$. This is a Gamma density, and $C=1/[s^{r}\Gamma(r)]$. It has mean
$\mu=rs>0$ and variance $rs^{2}.$\ If one wants to make $Z$ centered, the
density takes the form $\rho_{\ast}(z)=C(z+\mu)^{r-1}e^{-(z+\mu)/s}$ for
$-\mu<z<\infty$. For this case, $g_{\ast}(z)=s(z+\mu)_{+}$.

\item CASE 3. If $\deg g_{\ast}(x)=2$ and $g_{\ast}$ has two real roots,
$\rho_{\ast}$ can be written in the form $\rho_{\ast}(x)=Cx^{r-1}(1-x)^{s-1}$
for $0<x<1$, with parameters $r,s>0$. This is a Beta density, and
$C=1/\beta(r,s)$. It has mean $\mu=r/(r+s)>0$ and variance $rs/[(r+s)^{2}%
(r+s+1)]$. Centering the density gives $\rho_{\ast}(x)=C(x+\mu)^{r-1}%
(1-x-\mu)^{s-1}$ for $-\mu<x<1-\mu$. For this case, $g_{\ast}(x)=\frac{1}%
{r+s}(x+\mu)(1-x-\mu)$ when $-\mu<x<1-\mu$ and $0$ elsewhere.

\item CASE 4. If $\deg g_{\ast}(x)=2$ and $g_{\ast}$ has exactly one real
root, $\rho_{\ast}$ can be written in the form $\rho_{\ast}(x)=Cx^{-r}%
e^{-s/x}$ for $0<x<\infty$, with parameters $r>1$ and $s\geq0$. The
normalization constant is $C=\frac{s^{r-1}}{\Gamma(r-1)}$. If $r>2$, it has
mean $\mu=s/(r-2)\geq0$. If $r>3$, it has variance $s^{2}\Gamma(r-3)/\Gamma
(r-1)$. Centering this density yields $\rho_{\ast}(x)=C(x+\mu)^{-r}%
e^{-s/(x+\mu)}$ for $-\mu<x<\infty$, and assume that $r>3$. For this case,
$g_{\ast}(x)=\frac{1}{r-2}(x+\mu)^{2}$ when $-\mu<x$ and $0$ elsewhere.

\item CASE 5. If $\deg g_{\ast}(x)=2$ and $g_{\ast}$ has no real roots,
$\rho_{\ast}(x)=C\left(  1+x^{2}\right)  ^{-r}e^{s(\arctan x)}$ for
$-\infty<x<\infty$, with parameters $r>1/2$ and $-\infty<s<\infty$. The
normalization constant is $C=\frac{\Gamma(r)}{\sqrt{\pi}\Gamma(r-1/2)}%
\left\vert \frac{\Gamma(r-is/2)}{\Gamma(r)}\right\vert ^{2}$. If $r>1$, it has
mean $\mu=s/\left[  2(r-1)\right]  $. If $r>3/2$, it has variance $\left[
4\left(  r-1\right)  ^{2}+s^{2}\right]  /\left[  4(r-1)^{2}(2r-3)\right]  $.
The centered form of the density is $\rho_{\ast}(x)=C\left[  1+\left(
x+\mu\right)  ^{2}\right]  ^{-r}e^{s(\arctan(x+\mu))}$ for $-\infty<x<\infty$,
and assuming that $r>3/2$. For this case, $g_{\ast}(x)=\frac{1}{2(r-1)}\left[
1+\left(  x+\mu\right)  ^{2}\right]  $. Using our original notation,
$\alpha=\frac{1}{2(r-1)}$, $\beta=\frac{\mu}{r-1}$ and $\gamma=\frac{\mu
^{2}+1}{2(r-1)}$.
\end{itemize}

\subsection{Other Lemmas}

\begin{lemma}
\label{lem:phitail}Let $Z$ be a centered Pearson random variable. Then there
exist constants $k_{u}>k_{l}>0$ depending only on $\alpha,\beta,\gamma$ such
that when $z$ is large enough, we have the following inequalities.

\begin{enumerate}
\item If $\alpha=0$ and $\beta>0$,
\[
\frac{k_{l}}{z^{1+\gamma/\beta^{2}}e^{z/\beta}}\leq\Phi_{\ast}\left(
z\right)  \leq\frac{k_{u}}{z^{1+\gamma/\beta^{2}}e^{z/\beta}}\text{.}%
\]

\item If $\alpha>0$, when $z$ is large enough,
\[
\frac{k_{l}}{z^{1+1/\alpha}}\leq\Phi_{\ast}\left(  z\right)  \leq\frac{k_{u}%
}{z^{1+1/\alpha}}\text{.}%
\]

\item[3.] Assuming $Z$'s support extends to $-\infty$, if $\alpha>0$, when
$z<0$ and $\left\vert z\right\vert $ is large enough,
\[
\frac{k_{l}}{\left\vert z\right\vert ^{1+1/\alpha}}\leq1-\Phi_{\ast}\left(
z\right)  \leq\frac{k_{u}}{\left\vert z\right\vert ^{1+1/\alpha}}\text{.}%
\]

\end{enumerate}
\end{lemma}

\begin{proof}
For the proof of this lemma, which is presumably well-known, but is included
for completeness, we will use $C$ for the normalization constant of each
density to be considered.

In Point 1, let $\mu=\gamma/\beta>0$. Then $Z$ has support $\left(
-\mu,\infty\right)  $; see Case 2 in Appendix \ref{PearsonList}. In its
support, $Z$ has $g_{\ast}\left(  z\right)  =\beta z+\gamma=\beta\left(
z+\mu\right)  $ and density
\[
\rho_{\ast}\left(  z\right)  =C\left(  z+\mu\right)  ^{-\mu/\beta-1}%
\exp\left(  -\frac{z+\mu}{\beta}\right)  \text{.}%
\]
Note that
\[
\lim_{z\rightarrow\infty}z^{\mu/\beta}e^{z/\beta}g_{\ast}\left(  z\right)
\rho_{\ast}\left(  z\right)  =C\beta\lim_{z\rightarrow\infty}\frac
{z^{\mu/\beta}}{\left(  z+\mu\right)  ^{\mu/\beta}}\exp\left(  \frac{z}{\beta
}-\frac{z+\mu}{\beta}\right)  =C\beta e^{-\mu/\beta}.
\]
From Lemma \ref{lem:phibound},
\[
\frac{z-\beta}{z^{2}+\gamma}g_{\ast}\left(  z\right)  \rho_{\ast}\left(
z\right)  \leq\Phi_{\ast}\left(  z\right)  \leq\frac{1}{z}g_{\ast}\left(
z\right)  \rho_{\ast}\left(  z\right)
\]
so
\[
C\beta e^{-\mu/\beta}\leq\liminf_{z\rightarrow\infty}z^{1+\mu/\beta}%
e^{z/\beta}\Phi_{\ast}\left(  z\right)  \leq\limsup_{z\rightarrow\infty
}z^{1+\mu/\beta}e^{z/\beta}\Phi_{\ast}\left(  z\right)  \leq C\beta
e^{-\mu/\beta}\text{.}%
\]
Therefore, we can choose some constants $k_{u}\left(  \beta,\gamma\right)
>c_{l}\left(  \beta,\gamma\right)  >0$ such that when $z$ is large enough,
\[
\frac{k_{l}}{z^{1+\mu/\beta}e^{z/\beta}}\leq\Phi_{\ast}\left(  z\right)
\leq\frac{k_{u}}{z^{1+\mu/\beta}e^{z/\beta}}\text{.}%
\]

To prove Point 2, we first show that $\lim_{z\rightarrow\infty}z^{1/\alpha
}g_{\ast}\left(  z\right)  \rho_{\ast}\left(  z\right)  $ is a finite number
$K.$ We consider the cases $4\alpha\gamma-\beta^{2}=0$ and $4\alpha
\gamma-\beta^{2}>0$ separately. We need not consider $4\alpha\gamma-\beta
^{2}<0$ since it corresponds to Case 3 in Appendix \ref{PearsonList} for which
the right endpoint of the support of $Z$ is $b<\infty$ and so necessarily
$\alpha<0$.

Supose that $4\alpha\gamma-\beta^{2}=0$ and let $\mu=\frac{\beta}{2\alpha}>0$.
Then $\alpha z^{2}+\beta z+\gamma=\alpha\left(  z+\mu\right)  ^{2}$ has one
real root and the support of $Z$ is $\left(  -\mu,\infty\right)  $; see Case 4
in Appendix \ref{PearsonList}. In its support, $Z$ has $g_{\ast}\left(
z\right)  =\alpha\left(  z+\mu\right)  ^{2}$ and density
\[
\rho_{\ast}\left(  z\right)  =C\left(  z+\mu\right)  ^{-2-1/\alpha}\exp\left(
-\frac{s}{z+\mu}\right)
\]
where $s=\mu/\alpha=\beta/\left(  2\alpha^{2}\right)  $. Therefore,
\[
\lim_{z\rightarrow\infty}z^{1/\alpha}g_{\ast}\left(  z\right)  \rho_{\ast
}\left(  z\right)  =C\alpha\lim_{z\rightarrow\infty}\frac{z^{1/\alpha}%
}{\left(  z+\mu\right)  ^{1/\alpha}}\exp\left(  -\frac{s}{z+\mu}\right)
=C\alpha\text{.}%
\]
Now suppose that $\delta^{2}:=\left(  4\alpha\gamma-\beta^{2}\right)  /\left(
4\alpha^{2}\right)  >0$ so $\alpha z^{2}+\beta z+\gamma$ has two imaginary
roots and the support of $Z$ is $\left(  -\infty,\infty\right)  $. Letting
$\mu=\beta/\left(  2\alpha\right)  $ allows us to write $g_{\ast}\left(
z\right)  =\alpha\left(  z+\mu\right)  ^{2}+\alpha\delta^{2}$ and the density
of $Z$ as
\[
\rho_{\ast}\left(  z\right)  =C\left[  \left(  z+\mu\right)  ^{2}+\delta
^{2}\right]  ^{-1-\frac{1}{2\alpha}}\exp\left[  \frac{\mu}{\alpha\delta
}\arctan\left(  \frac{z+\mu}{\delta}\right)  \right]  \text{,}%
\]
a slight variation of the density in Case 5 in Appendix \ref{PearsonList}.
Note than in our present case,
\[
\lim_{z\rightarrow\infty}z^{1/\alpha}g_{\ast}\left(  z\right)  \rho_{\ast
}\left(  z\right)  =C\alpha\lim_{z\rightarrow\infty}\frac{z^{1/\alpha}%
}{\left[  \left(  z+\mu\right)  ^{2}+\delta^{2}\right]  ^{\frac{1}{2\alpha}}%
}\exp\left[  \frac{\mu}{\alpha\delta}\arctan\left(  \frac{z+\mu}{\delta
}\right)  \right]  =C\alpha\exp\left[  \frac{\mu\pi}{2\alpha\delta}\right]
\text{.}%
\]
From Lemma \ref{lem:phibound},
\[
\frac{\left(  1-2\alpha\right)  z-\beta}{\left(  1-\alpha\right)  z^{2}%
+\gamma}g_{\ast}\left(  z\right)  \rho_{\ast}\left(  z\right)  \leq\Phi_{\ast
}\left(  z\right)  \leq\frac{1}{z}g_{\ast}\left(  z\right)  \rho_{\ast}\left(
z\right)  \text{.}%
\]
From these bounds we conclude
\[
K\frac{1-2\alpha}{1-\alpha}\leq\liminf_{z\rightarrow\infty}z^{1+1/\alpha}%
\Phi_{\ast}\left(  z\right)  \leq\limsup_{z\rightarrow\infty}z^{1+1/\alpha
}\Phi_{\ast}\left(  z\right)  \leq K\text{.}%
\]
Therefore, when $z$ is large enough, for some constants $k_{u}\left(
\alpha,\beta,\gamma\right)  >k_{l}\left(  \alpha,\beta,\gamma\right)  >0$,
\[
\frac{k_{l}}{z^{1+1/\alpha}}\leq\Phi_{\ast}\left(  z\right)  \leq\frac{k_{u}%
}{z^{1+1/\alpha}}\text{.}%
\]

To prove Point 3, we consider Case 5 again.
\[
\lim_{z\rightarrow-\infty}\left\vert z\right\vert ^{1/\alpha}g_{\ast}\left(
z\right)  \rho_{\ast}\left(  z\right)  =C\alpha\lim_{y\rightarrow\infty}%
\frac{y^{1/\alpha}}{\left[  \left(  -y+\mu\right)  ^{2}+\delta^{2}\right]
^{\frac{1}{2\alpha}}}\exp\left[  \frac{\mu}{\alpha\delta}\arctan\left(
\frac{-y+\mu}{\delta}\right)  \right]  =C\alpha\exp\left[  -\frac{\mu\pi
}{2\alpha\delta}\right]  \text{.}%
\]
The conclusion follows similarly after using Lemma \ref{lem:phibound} when
$z<0$:
\[
\frac{\left(  1-2\alpha\right)  \left\vert z\right\vert -\beta}{\left(
1-\alpha\right)  \left\vert z\right\vert ^{2}+\gamma}g_{\ast}\left(  z\right)
\rho_{\ast}\left(  z\right)  \leq1-\Phi_{\ast}\left(  z\right)  \leq\frac
{1}{\left\vert z\right\vert }g_{\ast}\left(  z\right)  \rho_{\ast}\left(
z\right)  \text{.}%
\]

\end{proof}

\begin{lemma}
\label{lem:PearsonMoment}Let $Z$ be a centered Pearson random variable. If
$\alpha\leq0$, all moments of positive order exist. If $\alpha>0$, the moment
of order $m$ exists if and only if $m<1+1/\alpha$.
\end{lemma}

\begin{proof}
The random variables in Case 1 ($\alpha=\beta=0$) of Appendix
\ref{PearsonList} are Normal, while those in Case 3 ($\alpha<0$) have finite
intervals for support.\ It suffices to consider the cases where $\alpha=0$ and
$\beta>0$, and where $\alpha>0$. Let $m>0$. We will use the fact that
$\mathbf{E}\left[  \left|  Z\right|  ^{m}\right]  <\infty$ if and only if
$\sum_{n=1}^{\infty}n^{m-1}\mathbf{P}\left[  \left|  Z\right|  \geq n\right]
<\infty$.

If $\alpha=0$ and $\beta>0$, and $Z$ is supported over $\left(  a,\infty
\right)  $, then by Lemma \ref{lem:phitail}, $\mathbf{E}\left[  \left|
Z\right|  ^{m}\right]  <\infty$ if and only if
\[
\sum_{n=1}^{\infty}\frac{n^{m-1}}{n^{1+\gamma/\beta^{2}}e^{n/\beta}}%
<\infty\text{,}%
\]
which is always the case.

Now suppose $\alpha>0$. Since $\mathbf{P}\left[  \left\vert Z\right\vert \geq
n\right]  =\Phi_{\ast}\left(  n\right)  +1-\Phi_{\ast}\left(  -n\right)  $,
the$\circ$n by Lemma \ref{lem:phitail} again, $\mathbf{E}\left[  \left\vert
Z\right\vert ^{m}\right]  <\infty$ if and only if
\[
\sum_{n=1}^{\infty}\frac{n^{m-1}}{n^{1+1/\alpha}}=\sum_{n=1}^{\infty}\frac
{1}{n^{2+1/\alpha-m}}\mathbf{<\infty}\text{.}%
\]
This is the case if and only if $2+1/\alpha-m>1$, i.e. $m<1+1/\alpha$.
\end{proof}

\end{document}